\documentclass{daj}

\usepackage{amsmath,amsthm,amsfonts,amssymb,times}

\usepackage{amssymb}
\usepackage{amscd}
\usepackage{amsfonts}
\usepackage{version}
\usepackage{endnotes}
\usepackage{float}

\let\footnote=\endnote

\usepackage{graphicx}

\newtheorem{theorem}{Theorem}[section]
\newtheorem{lemma}[theorem]{Lemma}
\newtheorem{proposition}[theorem]{Proposition}
\newtheorem{corollary}[theorem]{Corollary}

\newtheorem*{ramsey-lemma}{Proposition \ref{ramsey-prop}}
\newtheorem*{counting-lem-repeat}{Proposition \ref{counting-lem}}
\newtheorem*{distribution-repeat}{Proposition \ref{distribution}}

\theoremstyle{definition}

\newtheorem{question}{Question}
\newtheorem{definition}{Definition}[section]

\renewcommand{\leq}{\leqslant}
\renewcommand{\geq}{\geqslant}

\newcommand\QM{\operatorname{QM}}

\newcommand\id{\operatorname{id}}

\newcommand\trig{\operatorname{trig}}
\def\F{\mathbb{F}}
\def\R{\mathbb{R}}
\def\C{\mathbb{C}}
\def\Z{\mathbb{Z}}
\def\E{\mathbb{E}}

\def\Q{\mathbb{Q}}

\def\G{\mathbb{G}}
\def\K{\mathbb{K}}

\def\N{\mathbb{N}}

\def\eps{\varepsilon}

\numberwithin{equation}{section}

\dajAUTHORdetails{%
  title = {Monochromatic sums and products}, 
  author = {Ben Green and Tom Sanders},
  plaintextauthor = {Ben Green and Tom Sanders},
  plaintexttitle = {Monochromatic sums and products}, 
  keywords = {arithmetic ramsey, sums, products, monochromatic},
}   

\dajEDITORdetails{%
   year={2016},
   number={5},
   received={30 October 2015},   
   published={28 February 2016},  
   doi={10.19086/da.613},       
}   

\begin{document}

\begin{frontmatter}[classification=text]

\title{Monochromatic sums and products} 

\author[bg]{Ben Green\thanks{Supported by ERC Starting Grant 279438 \emph{Approximate algebraic structure and applications}, and by a Simons Investigator Grant.}}
\author[ts]{Tom Sanders}

\begin{abstract}
Suppose that $\F_p$ is coloured with $r$ colours. Then there is some colour class containing at least $c_r p^2$ quadruples of the form $(x, y , x + y, xy)$.
\end{abstract}
\end{frontmatter}

\section{Introduction and notation}

The following beautiful question was asked on numerous occasions by Hindman (see, for example, \cite[Question 3]{hindmanleaderstrauss}) and is very well-known.

\begin{question}\label{q1}
Suppose that the natural numbers $\N$ are finitely coloured. Do there exist $x$ and $y$ such that $x, y, x+y$ and $xy$ all have the same colour? 
\end{question}

It follows from Schur's theorem \cite[Hilfssatz]{schur} that the answer is affirmative if either $x+y$ or $xy$ is omitted from the list.  (In the latter case it is an observation of Graham \cite[p1]{hindman} that we can consider the colouring of $\N$ induced on $\{2^n: n \in \N\}$.) It is also known that the answer to Question \ref{q1} is affirmative for 2-colourings; in fact, every 2-colouring of $\{1,2,\dots, 252\}$ contains a monochromatic quadruple of the stated type, as established by Graham \cite[Theorem 4.3]{hindman}. In general, however, Question \ref{q1} is quite open, as indeed is the following considerably weaker statement.  (See the remarks following \cite[Question 3]{hindmanleaderstrauss}.)

\begin{question}\label{q2}
Suppose that the natural numbers $\N$ are finitely coloured. Do there exist $x$ and $y$, not both $2$, such that $x+y$ and $xy$ all have the same colour?
\end{question}

In this paper we are concerned with so-called finite field analogues of the above questions where we replace $\N$ by the finite fields $\F_p$. Schur's theorem was originally designed for application to $\F_p^*$ so it should be of little surprise that its finite field analogue is routine.  Shkredov seems to have been the first to address the finite field analogue of Question \ref{q2} in \cite[Theorem 1.2]{shkredov} (later generalised by Cilleruelo to any finite field \cite[Corollary 4.2]{cilleruelo}) and in fact he proves rather more, namely the following \cite[Theorem 1.2]{shkredov}.

\begin{theorem}[Shkredov]\label{shkredov}
Suppose that $A\subset \F_p$ has size $\alpha p$. Then there are at least $c_\alpha p^2$ triples $(x,x+y,xy)$ in $A^3$, for some $c_{\alpha} > 0$ which does not depend on $p$.
\end{theorem}

(Technically, Shkredov only states that there is at least one such triple, not $\geq c_{\alpha} p^2$, but it is obvious that his proof gives this stronger statement.) It follows trivially that if $\F_p$ is $r$-coloured and $p$ is sufficiently large in terms of $r$ then there are non-zero elements $x$ and $y$ such that $x,x+y$ and $xy$ all have the same colour.  A corresponding result for colourings over $\Q$ (or any countable field) was obtained by Bergelson and Moreira \cite[Theorem 1.2]{bergelson-moreira} using ergodic-theoretic techniques.

In this paper we solve the finite field analogue of Question \ref{q1} by establishing the following.
\begin{theorem}\label{mainthm}
Suppose that $\F_p$ is $r$-coloured. Then there are at least $c_r p^2$ monochromatic quadruples $(x,y,x+y,xy)$, where $c_r > 0$ does not depend on $p$.
\end{theorem}
Note that while Theorem \ref{shkredov} is a \emph{density result}, there is no such version of Theorem \ref{mainthm}, as can be seen by considering the set $\{ x \in \F_p : \frac{p}{3} < x < \frac{2p}{3}\}$. This has density roughly $\frac{1}{3}$ but does not even contain a set of the form $\{x, y, x+y\}$.

Before diving into the remains of the paper it may be useful to say that we discuss the outline of the main argument in \S\ref{sec4}, after setting up some notation in \S\ref{sec3}.  \S\ref{sec2} also includes some notation but more importantly develops the tools for counting the configurations we are interested in.  After that the remaining sections fill out the details of \S\ref{sec4}.

\emph{Notation.} We use fairly standard asymptotic notation such as $O()$, $o()$, $\gg$ and $\ll$. We write $o_{M; p \rightarrow \infty}(1)$ to mean a quantity tending to $0$ as $p \rightarrow \infty$, but in a manner that may depend on the parameter $M$. Similarly, for example, $O_{\eps}(1)$ denotes a constant which may depend on some parameter $\eps$. Throughout the paper $\F$ will denote the finite field with $p$ elements (we do not explicitly indicate the prime $p$). Occasionally we shall write $\mu_{\F}$ for the uniform probability measure on $\F$. As is standard in additive combinatorics we write $\E_{x \in X} = \frac{1}{|X|} \sum_{x \in X}$ for averages over some finite set $X$.

\section{Counting quadruples}\label{sec2}

Given functions $f_1,f_2,f_3,f_4:\F \rightarrow \C$ we write
\begin{equation*}
T(f_1,f_2,f_3,f_4):=\E_{x,y \in \F}{f_1(x)f_2(y)f_3(x+y)f_4(xy)},
\end{equation*}
so that if $A \subset \F$ then $p^2T(1_A,1_A,1_A,1_A)$ is the number of quadruples $(x,y,x+y,xy) \in A^4$.\vspace*{8pt}

\emph{The additive Fourier transform and counting sums.}  The quantity $p^2T(1_A,1_A,1_A,1_\F)$ is the number of triples $(x,y,x+y)\in A^3$ and this is well understood through the \emph{additive} Fourier transform. We write $\widehat{\F}$ for the dual group of the additive group of $\F$, and given $f:\F \rightarrow \C$ and $\gamma \in \widehat{\F}$ define the (additive) Fourier transform of $f$ by
\begin{equation*}
\widehat{f}(\gamma):=\E_{x \in \F}{f(x)\overline{\gamma(x)}}.
\end{equation*}
Writing $e_p(x):=\exp(2\pi i x/p)$, we know that the elements of $\widehat{\F}$ are just the maps of the form $x \mapsto e_p(rx)$ as $r$ ranges over $\F$ and so we shall frequently identify $\widehat{\F}$ with $\F$ and write $\widehat{f}(r)$ for $\widehat{f}(\gamma)$ where $\gamma(x)=e_p(rx)$ for all $x \in \F$. As usual the key tools are the inversion formula
\begin{equation}\label{eqn.ainv}
f(x) =\sum_{r \in \F}{\widehat{f}(r)e_p(rx)} \text{ for all } x \in \F,
\end{equation}
and Parseval's theorem
\begin{equation}\label{eqn.apars}
\|f\|_{2}^2=\E_{x \in \F}{|f(x)|^2} = \sum_{r \in \F}{|\widehat{f}(r)|^2}.
\end{equation}
The convolution of two functions $f,g:\F \rightarrow \C$ is defined to be
\begin{equation*}
f \ast g(y):=\E_x{f(x)g(y - x)} \text{ for all } y \in \F.
\end{equation*}
It is well-known and easy to prove that
\begin{equation*}
\widehat{f \ast g}(r) = \widehat{f}(r)\widehat{g}(r) \text{ for all }r \in \F.
\end{equation*}
Finally, define the $u^+_2$-norm of a function $f:\F \rightarrow \C$ by
\begin{equation*}
\|f\|_{u_2^+}:=\sup\{|\langle f,\gamma\rangle| : \gamma \in \widehat{\F}\}=\sup\{|\E_x{f(x)\overline{e_p(rx)}}|:r \in \F\}.
\end{equation*}
Usually, the $u_2^+$-norm is simply denoted $u_2$; we include the plus sign because we shall shortly need the multiplicative analogue. It is easy to see that $\|\cdot\|_{u_2^+}$ is a norm and that it is dominated by $\|\cdot \|_{1}$.  This norm is important because of the following lemma.
\begin{proposition}\label{prop.add}
Suppose that $f_1,f_2,f_3:\F \rightarrow \C$ are such that $\|f_1\|_{2},\|f_2\|_{2},\|f_3\|_{2}\leq 1$.  Then
\begin{equation*}
|T(f_1,f_2,f_3,1_{\F})| \leq \inf_i\|f_i\|_{u_2^+}.
\end{equation*}
\end{proposition}
\begin{proof}
This is very standard. By the inversion formula \eqref{eqn.ainv} we have
\[ T(f_1, f_2, f_3, 1_{\F}) = \sum_r \hat{f}_3(r) \hat{f}_1(-r) \hat{f}_2(-r).\]
The stated inequality now follows from the Cauchy-Schwarz inequality and Parseval's theorem \eqref{eqn.apars}. For example,
\begin{align*}
|T&(f_1, f_2, f_3, 1_{\F})| \leq \sup_r |\hat{f}_3(r)| \left(\sum_r |\hat{f}_1(-r)| |\hat{f}_2(-r)|\right)  = \Vert f_3 \Vert_{u_2^+} \sum_r |\hat{f}_1(-r)| |\hat{f}_2(-r)| \\ & \leq \Vert f_3 \Vert_{u_2^+} \big( \sum_r |\hat{f}_1(-r)|^2 \big)^{1/2} \big( \sum_r |\hat{f}_2(-r)|^2 \big)^{1/2}  = \Vert f_3 \Vert_{u_2^+} \Vert f_1 \Vert_2 \Vert f_2 \Vert_2 \leq \Vert f_3 \Vert_{u_2^+},
\end{align*}
with almost identical proofs being available to bound the left hand side by $\|f_1\|_{u_2^+}$ and $\|f_2\|_{u_2^+}$.
\end{proof}

\emph{Multiplicative characters and counting products.} The quantity $p^2 T(1_A,1_A,1_\F,1_A)$ counts the number of triples $(x,y,xy) \in A^3$ and this is well understood through the \emph{multiplicative} Fourier transform. If $\chi : \F^* \rightarrow \C$ is a character, we extend $\chi$ to all of $\F$ by setting $\chi(0) = 1$. We write $\widehat{\F^*}$ for the set of all such extended characters and then define the $u_2^\times$-semi-norm of a function $f:\F \rightarrow \C$ to be
\begin{equation*}
\|f\|_{u_2^\times}:=\sup \{ |\langle f, \chi \rangle|: \chi \in \widehat{\F^*}\}
\end{equation*}
where $\langle f, \chi \rangle := \E_{x \in \F}{f(x)\overline{\chi(x)}}$.  Note that this is only a semi-norm because functions $f$ with $f(0)=-f(1)$ and $f(x)=0$ elsewhere have $\|f\|_{u_2^\times}=0$.

The analogue of Proposition \ref{prop.add} is then the following. 
\begin{proposition}\label{prop.mult}
Suppose that $g_1,g_2,g_4:\F \rightarrow \C$ are functions such that 
\begin{itemize}
\item $\Vert g_1 \Vert_{\infty}, \Vert g_2 \Vert_{\infty}, \Vert g_4 \Vert_{\infty} \leq K$ for some $K \geq 1$ and
\item $\Vert g_1 \Vert_2, \Vert g_2 \Vert_2, \Vert g_4 \Vert_2 \leq 1$.
\end{itemize}
Then
\begin{equation*}
|T(g_1,g_2,1_{\F},g_4)| \leq \inf_i{\|g_i\|_{u_2^\times}} + \frac{4K^3}{p}.
\end{equation*}
\end{proposition}
\begin{proof}
Introduce the auxiliary quantity 
\[ \tilde T(g_1, g_2,g_4) := \E_{x,y \in \F^*} g_1(x) g_2(y) g_4(xy).\]
By exactly the same analysis as in the proof of Proposition \ref{prop.add}, but using the (multiplicative) Fourier transform on $\F^*$ instead of the additive one, we obtain
\begin{equation}\label{mult-an} |\tilde T(g_1, g_2, g_4) | \leq  \frac{p}{p-1}\sup_{\chi} |\E_{x \in \F^*} g_i(x) \overline{\chi(x)}|\end{equation} for each $i \in\{ 1,2,4\}$. The extra factor of $p/(p-1)$ comes from the fact that
\begin{equation*}
\|g_j\|_{L_2(\mu_{\F^*})}^2\leq \frac{p}{p-1}\|g_j\|_{L_2(\mu_\F)}^2 \text{ for all } j \in \{1,2,4\}.
\end{equation*}
Now we have
\begin{align*}
T(g_1, g_2, 1_{\F}, g_4) & = -\frac{1}{p^2} g_1(0) g_2(0) g_4(0) + g_1(0) g_4(0) \frac{1}{p^2}\sum_y g_2(y)\\
& \qquad + g_2(0) g_4(0) \frac{1}{p^2}\sum_x g_1(x)+ \left(\frac{p-1}{p}\right)^2 \tilde T(g_1, g_2, g_4).\end{align*}
The first three terms may be estimated somewhat trivially using the bound $\Vert g_i \Vert_{\infty} \leq K$; in magnitude they total at most $\frac{3K^3}{p}$. From this and \eqref{mult-an} we obtain
\[ |T(g_1, g_2, 1_{\F}, g_4)| \leq \frac{3K^3}{p} +  \frac{p-1}{p} \sup_{\chi}|\E_{x \in \F^*} g_i(x) \overline{\chi}(x)|,\] for each $i = 1,2,4$. Noting that 
\[ \E_{x \in \F^*} g_i(x) \overline{\chi(x)} = \frac{p}{p-1}\langle g_i,\chi\rangle - \frac{1}{p-1} g_i(0),\] the result follows easily. 
\end{proof}
 
\emph{Counting sums and products.}\label{subsec.csp}
In the light of Proposition \ref{prop.add} and Proposition \ref{prop.mult} one might hope that $|T(f_1,f_2,f_3,f_4)|$ is controlled by the combination of $\|f_i\|_{u_2^+}$ and $\|f_i\|_{u_2^\times}$.  Unfortunately this is not the case, as the following example shows.  Given $\gamma \in \widehat{\F}$ and $\chi \in \widehat{\F^*}$ let
\begin{equation*}
f_1(t):=\gamma(t^2)\chi(t), f_2(t):=\gamma(t^2)\chi(t),
\end{equation*}
and
\begin{equation*}
f_3(t):=\gamma(-t^2), f_4(t):=\gamma(2t)\overline{\chi(t)}.
\end{equation*}
Then
\begin{eqnarray*}
T(f_1,f_2,f_3,f_4) & = &\E_{x,y \in \F}{\gamma(x^2 + y^2 - (x+y)^2+2xy)\chi(x)\chi(y)\overline{\chi(xy)}}\\
& = & \frac{(p-1)^2+1}{p^2} = 1+o_{p\rightarrow \infty}(1).
\end{eqnarray*}
On the other hand if $\gamma$ and $\chi$ are non-trivial then it may be checked using character sums estimates of the type discussed at the beginning of \S\ref{sec8} that $\Vert f_i \Vert_{u_2^+}, \Vert f_i \Vert_{u_2^{\times}} \ll p^{-1/2}$. (As this is only a motivating example, the exact details of the proof of this need not concern us.)

Write $Q(\F)$ for the set of quadratic phases on $\F$, that is to say
\begin{equation*}
Q(\F):=\{x \mapsto e_p(rx^2+sx):  r,s \in \F\}.
\end{equation*}
In the above example, these quadratic phases mixed with multiplicative characters. This suggests the following definition, which is a key definition in our paper.

\begin{definition}[$\QM$-norm]\label{def.qmnorm}
Suppose that $f : \F \rightarrow \C$ is a function. Then we define
\begin{equation*}
\|f\|_{\QM}:=\sup  \{ |\langle f, \phi\chi \rangle |: \phi \in Q(\F), \chi \in \widehat{\F^*}\},
\end{equation*}
which is easily seen to be a norm.
\end{definition}

The reason this is such an important definition for us is the following fact, which is the main result of the section. It says that in a sense these quadratic-multiplicative examples are the \emph{only} ones affecting the count $T(f_1, f_2,f_3,f_4)$. 

\begin{proposition}\label{gvn-nonlinear}
Suppose that $f_1,f_2,f_3,f_4:\F \rightarrow \C$ are such that $\|f_1\|_\infty,\|f_2\|_\infty,\|f_3\|_\infty,\|f_4\|_\infty \leq 1$. Then
\begin{equation*}
|T(f_1,f_2,f_3,f_4)| = O(\inf_i\max (p^{-1/64},\|f_i\|_{\QM}^{1/5})).
\end{equation*}
\end{proposition}

We shall begin the proof of this shortly, but first we must introduce one final norm.\vspace*{8pt}

\emph{The $u^+_3$-norm.}
Higher-order variants of the $u_2^+$-norm examining correlation with quadratic phases have been widely-studied since the ground-breaking paper \cite{gowers-4ap} of Gowers.  We shall only need a fragment of his ideas here. We define the $u^+_3$-norm by
\begin{equation*}
\|f\|_{u^+_3} :=\sup\{|\langle f,\phi \rangle| : \phi \in Q(\F)\}=\sup\{|\E_{x \in \F}{f(x)\overline{\phi(x)}}|:\phi \in Q(\F)\}.
\end{equation*}
We have the chain of inequalities
\begin{equation}\label{chain} \Vert \cdot \Vert_{u_2^+} \leq \Vert \cdot \Vert_{u_3^+} \leq \Vert \cdot \Vert_{\QM} \leq \Vert \cdot \Vert_1. \end{equation}

A key ingredient in the proof of Proposition \ref{gvn-nonlinear} is the following result which, in view of \eqref{chain}, is in fact stronger than that result when $i = 3$. 
\begin{proposition}\label{gvn-3-l2}
Suppose that $f_1,f_2,f_3,f_4 : \F \rightarrow \C$ are such that $\Vert f_1 \Vert_{\infty}, \Vert f_2 \Vert_{\infty}, \Vert f_4 \Vert_{\infty} \leq 1$, and such that $f_3$ satisfies the slightly weaker bounds $\Vert f_3 \Vert_2 \leq 1$ and $\Vert f_3 \Vert_{\infty} \leq p^{1/16}$. Then
\begin{equation*}
|T(f_1,f_2, f_3, f_4) |^8 \leq  \|f_3\|_{u^+_3}^2+ O(p^{-1/2}).\end{equation*}
\end{proposition}
\emph{Remark.} Of course, the conclusion is valid under the stronger assumption that $\Vert f_3 \Vert_{\infty} \leq 1$, but it will be helpful to allow weaker bounds when applying Lemma \ref{decomposition-quadratic} below.

The proof follows the ideas of \cite{gowers-4ap}, although the additional multiplicative structure makes the argument considerably simpler. In particular we make no use of the Balog-Szemer\'edi-Gowers theorem or any results in the direction of Freiman's theorem. Following Gowers we introduce some notation.  For any $f:\F \rightarrow \C$, we write
\begin{equation*}
\Delta_hf(x):=f(x+h)\overline{f(x)} \text{ for all }x,h \in \F.
\end{equation*}
The operator acts as a difference operator in the exponent so that if $\phi \in Q(\F)$ then $\Delta_h \phi$ is (a constant times) a linear character, and that character itself depends linearly on $h$. As in \cite{gowers-4ap}, we use a fairly straightforward converse of this fact. 
\begin{lemma}\label{lem.tim}
Suppose that $f:\F \rightarrow \C$ has $\|f\|_{2}\leq 1$. Then
\[ \sup_{h \in \F^*, r \in \F} \E_{z \in \F} |\widehat{\Delta_{zh}f}(zr)|^2 \leq \Vert f \Vert_{u_3^+}^2.\]
\end{lemma}
\begin{proof}
Let $h \in \F^*$ and $r \in \F$ be arbitrary. Since $h \neq 0$ we can write $g(x):=f(x)e_p(-rx^2/2h)$ and note that $\Vert g \Vert_2 = \Vert f \Vert_2 \leq 1$ and that $\Vert g \Vert_{u_3^+} = \Vert f \Vert_{u_3^+}$.
Applying the basic facts of additive Fourier analysis we have
\begin{align*}
\E_{z \in \F} |\widehat{\Delta_{zh}f}(zr)|^2 &  = \E_{z}{\E_{x,y}{f(x)\overline{f(x+zh)}e_p(zrx) \overline{f(y)}f(y+zh)e_p(-zry)}}\\
 & =  \E_{z}{\E_{x,y}{g(x)\overline{g(x+zh)}\overline{g(y)}g(y+zh)}}\\
 & = \E_{z'}{\E_{x,y}{g(x)\overline{g(x+z')}\overline{g(y)}g(y+z')}} \;\; \mbox{(since $h \neq 0$)}\\
 & =  \sum_s{|\widehat{g}(s)|^4}  
 \leq \|g\|_{u_2^+}^2\|g\|_{2}^2 \leq \|g\|_{u_2^+}^2 \leq \| g \|^2_{u_3^+} = \| f \|^2_{u_3^+}.
\end{align*}
The result follows.
\end{proof}

\begin{proof}[Proof of Proposition \ref{gvn-3-l2}]
If $z \in \F^*$ then, as $(x,y)$ ranges over $\F \times \F$, so does $(xz^{-1}, zy)$. It follows that 
\[ T(f_1, f_2, f_3, f_4) = \E_{x,y \in \F} f_1(xz^{-1})f_2(zy)f_3(xz^{-1}+zy)f_4(xy).\] Averaging over all $z$, we have
\[
T(f_1,f_2,f_3,f_4) = \E_{x,y \in \F} f_4(xy) \E_{z \in \F^*} f_1(xz^{-1})f_2(zy)f_3(xz^{-1}+zy).
\]
By Cauchy-Schwarz and the inequality $\|f_4\|_{\infty}\leq 1$ it follows that 
\begin{align}\nonumber
|T(f_1, f_2, f_3, f_4)|^2 & \leq \E_{x,y \in \F} |\E_{z \in \F^*} f_1(xz^{-1})f_2(zy)f_3(xz^{-1}+zy) |^2 \\ \nonumber  & = \E_{z_0, z_1 \in \F^*} \E_{x \in \F} f_1(xz_0^{-1})\overline{f_1(xz_1^{-1})}\\
& \qquad \qquad \qquad \times   \E_{y \in \F} f_2(z_0y)\overline{f_2(z_1y)} f_3(xz_0^{-1}+z_0y)\overline{f_3(xz_1^{-1}+z_1y)} .\label{1st-cauchy}
\end{align}
For fixed $z_0, z_1$, we apply Cauchy-Schwarz to the inner average over $x$.  We obtain
\begin{align*}
  \big| & \E_{x \in \F} f_1(xz_0^{-1})\overline{f_1(xz_1^{-1})}  \E_{y \in \F} f_2(z_0y)\overline{f_2(z_1y)}f_3(xz_0^{-1}+z_0y)\overline{f_3(xz_1^{-1}+z_1y)}  \big|^2 \\ & \leq 
\E_{x \in \F} |f_1 (x z_0^{-1}) f_1(x z_1^{-1})|^2  \cdot \E_{x \in \F} \big| \E_{y \in \F} f_2(z_0y)\overline{f_2(z_1y)}f_3(xz_0^{-1}+z_0y)\overline{f_3(xz_1^{-1}+z_1y)} \big|^2.
\end{align*}
Since $\Vert f_1 \Vert_{\infty} \leq 1$ (in fact $\Vert f_1 \Vert_4 \leq 1$ would be enough) this is bounded above by
\begin{align*}
\E_{x \in \F} & \big| \E_{y \in \F} f_2(z_0y)\overline{f_2(z_1y)}f_3(xz_0^{-1}+z_0y)\overline{f_3(xz_1^{-1}+z_1y)} \big|^2 \\ & = \E_{x,y_0,y_1 \in \F} f_2(z_0y_0)\overline{f_2(z_1y_0)}\overline{f_2(z_0y_1)}f_2(z_1y_1) \times \\ & \times f_3(xz_0^{-1}+z_0y_0)\overline{f_3(xz_1^{-1}+z_1y_0)}  \overline{f_3(xz_0^{-1}+z_0y_1)}f_3(xz_1^{-1}+z_1y_1).
\end{align*}
Substituting back in to \eqref{1st-cauchy}, we see that \begin{align*}
& |T(f_1, f_2, f_3, f_4)|^4  \leq \E_{z_0, z_1 \in \F^*, x, y_0, y_1 \in \F} f_2(z_0y_0)\overline{f_2(z_1y_0)}\overline{f_2(z_0y_1)}f_2(z_1y_1) \times \\ & \qquad  \times f_3(xz_0^{-1}+z_0y_0) \overline{f_3(xz_1^{-1}+z_1y_0)}  \overline{f_3(xz_0^{-1}+z_0y_1)}f_3(xz_1^{-1}+z_1y_1).
\end{align*}
Write $y = y_0$ and $h = y_1 - y_0$. The pair $(y,h)$ ranges uniformly over $\F \times \F$ as $(y_0, y_1)$ does. Evidently
\[ f_2(z_0y_0)\overline{f_2(z_1y_0)} \overline{f_2(z_0y_1)}f_2(z_1y_1) = f_2(z_0y)\overline{f_2(z_1y)}\overline{f_2(z_0(y + h))}f_2(z_1(y+h)); \]
moreover
\[ f_3(xz_0^{-1}+z_0y_0)\overline{f_3(xz_1^{-1}+z_1y_0)}  \overline{f_3(xz_0^{-1}+z_0y_1)}f_3(xz_1^{-1}+z_1y_1) = \] \[ \overline{\Delta_{z_0h}f_3(xz_0^{-1}+z_0y)}\Delta_{z_1h}f_3(xz_1^{-1}+z_1y) .\]
Thus
\[ |T(f_1, f_2, f_3, f_4)|^4 \leq \E_{z_0, z_1 \in \F^*, y, h \in \F} f_2(z_0y)\overline{f_2(z_1y)}\overline{f_2(z_0(y + h))}f_2(z_1(y+h)) \times \] \begin{equation}\label{to-cauchy-2} \qquad \qquad \times \E_{x \in \F}\overline{\Delta_{z_0h}f_3(xz_0^{-1}+z_0y)}\Delta_{z_1h}f_3(xz_1^{-1}+z_1y) .\end{equation}
If $z \in \F^*$, we write $m_z : L^{\infty}(\F) \rightarrow L^{\infty}(\F)$ for the operator defined by $(m_z g)(x) := g(zx)$ for all $x \in \F$. Then
\[ \overline{\Delta_{z_0h}f_3(xz_0^{-1}+z_0y)}\Delta_{z_1h}f_3(xz_1^{-1}+z_1y) = \overline{m_{z_0^{-1}}\Delta_{z_0 h} f _3(x + z_0^2 y)}m_{z_1^{-1}} \Delta_{z_1 h} f_3 (x + z_1^2 y).\]
Therefore
\[ \E_{x \in \F}\overline{\Delta_{z_0h}f_3(xz_0^{-1}+z_0y)}\Delta_{z_1h}f_3(xz_1^{-1}+z_1y) = (F_{z_0, h} \ast G_{z_1, h} )((z_0^2 - z_1^2) y),\] where
\[ F_{z_0, h}(t) := \overline{m_{z_0^{-1}}\Delta_{z_0 h}f_3(t)}\]
and
\[ G_{z_1, h}(t) :=  m_{z_1^{-1}}\Delta_{z_1 h} f_3(-t).\]
Therefore by \eqref{to-cauchy-2} we have
\[ |T(f_1, f_2, f_3, f_4)|^4 \leq  \E_{z_0, z_1 \in \F^*, y, h \in \F} f_2(z_0y)\overline{f_2(z_1y)} \overline{f_2(z_0(y + h))}f_2(z_1(y+h))  (F_{z_0, h} \ast G_{z_1, h} )((z_0^2 - z_1^2) y).\]
By Cauchy-Schwarz once more (and the fact that $\Vert f_2 \Vert_{\infty} \leq 1$) we have
\[ |T(f_1, f_2, f_3, f_4)|^8 \leq \E_{h \in \F} \E_{z_0, z_1 \in \F^*} \E_{y \in \F}|(F_{z_0, h} \ast G_{z_1, h} )((z_0^2 - z_1^2) y)|^2.\]
We have arranged the three averages in this way to make the next step clearer. If $z_0^2 \neq z_1^2$ then the inner average over $y$ is precisely $\Vert F_{z_0, h} \ast G_{z_1, h} \Vert_2^2$. If $z_0^2 = z_1^2$ then the inner average is of course simply $|(F_{z_0, h} \ast G_{z_1, h} )(0)|^2$. There are at most $2(p-1)$ pairs satisfying the second condition, so we have 
\begin{align}\nonumber |T(f_1, f_2, f_3, f_4)|^8 & \leq \E_{h \in \F}\E_{z_0, z_1 \in \F^*} \Vert F_{z_0, h} \ast G_{z_1, h} \Vert_2^2 +  \\ \nonumber & \qquad\qquad + \frac{2}{p-1} \sup_{z_0, z_1}|(F_{z_0, h} \ast G_{z_1, h} )(0)|^2 \\ \nonumber & \leq \E_{h \in \F, z_0, z_1 \in \F^*} \Vert F_{z_0, h} \ast G_{z_1, h} \Vert_2^2 + O(p^{-1/2}) \\ \label{use-9} & = \E_{h \in \F, z_0, z_1 \in \F^*} \sum_{r \in \F}| \widehat{F}_{z_0,h}(r)|^2 |\widehat{G}_{z_1,h}(r)|^2 + O(p^{-1/2}).\end{align} where in the second line we used the inequality $\Vert f_3 \Vert_{\infty} \leq p^{1/16}$, and in the third (the additive) Parseval's identity and the fact that convolution goes to multiplication. Now observe that for any $g \in L^{\infty}(\F)$ and $z \in \F^*$ we have
\[ (m_{z^{-1}} g)^{\wedge}(r)  = \E_{x \in \F} g(z^{-1} x) e_p(-rx) = \E_{x \in \F} g(x) e_p(-rzx) = \hat{g}(zr).\]
Thus
\[ \widehat{F}_{z_0, h}(r) = \overline{\widehat{\Delta_{z_0 h} f_3}(-z_0 r)} \text{ and } \widehat{G}_{z_1, h}(r) = \widehat{\Delta_{z_1 h} f_3}(-z_1 r).\]
It follows from \eqref{use-9} that
\begin{align*} |T(f_1, f_2, f_3, f_4)|^8 & \leq \E_{h \in \F, z_0, z_1 \in \F^*} \sum_{r \in \F} |\widehat{\Delta_{z_0h}f_3}(z_0r)|^2| \widehat{\Delta_{z_1h}f_3}(z_1r)|^2+O(p^{-1/2}).\\   & = \E_{h \in \F} \sum_{r \in \F}  \bigg(  \E_{z \in \F^*} | \widehat{\Delta_{zh} f_3}( zr) |^2  \bigg)^2 + O(p^{-1/2}) .\end{align*}
Using Parseval's identity, we see that the contribution to this average from $h = 0$ is
\begin{align*}   \frac{1}{p} \sum_{r \in \F} \bigg( \E_{z \in \F^*} |\widehat{\Delta_0 f_3}(zr) |^2 \bigg)^2  & = \frac{1}{p}\sum_{r \in \F} \bigg( \E_{z \in \F^*} \big| \widehat{|f_3|^2}(zr) \big| \bigg)^2 \\ & \leq \frac{4}{p^3} \sum_{r \in \F} \bigg( \sum_{z \in \F} \big| \widehat{|f_3|^2}(zr) \big|^2  \bigg)^2 \\ & \leq \frac{4}{p} \big| \widehat{ |f_3|^2 }(0) \big|^4 + \frac{4}{p^2} \bigg(  \sum_{s \in \F} \big| \widehat{|f_3|^2}(s) \big|^2 \bigg)^2 \\ & = \frac{4}{p} \big| \widehat{ |f_3|^2 }(0) \big|^4 + \frac{4}{p^2} \bigg( \E_{x \in \F} \big ||f_3|^2(x) \big|^2\bigg)^2 \\ & \leq \frac{4\Vert f_3 \Vert_{\infty}^8}{p} + \frac{4\Vert f_3 \Vert_{\infty}^8}{p^2}  = O(p^{-1/2}).\end{align*}
Thus by Parseval's identity (since $z$ ranges over $\F^*$) and interchanging the order of summation we have
\begin{align*} |T(& f_1, f_2, f_3, f_4)|^8 \\ & \leq \E_{h \in \F} 1_{\F^*}(h) \sum_{r \in \F}  \bigg(  \E_{z \in \F^*} | \widehat{\Delta_{zh} f_3}( zr) |^2  \bigg)^2 + O(p^{-1/2}) \\ & \leq \bigg(\sup_{h \in \F^*,r \in \F } \E_{z \in \F^*} | \widehat{\Delta_{zh} f_3}( zr) |^2\bigg) \cdot \E_{h \in \F}  \sum_{r \in \F}   \E_{z \in \F^*} | \widehat{\Delta_{zh} f_3}( zr) |^2   + O(p^{-1/2}) \\ & = \bigg(\sup_{h\in \F^*,r \in \F } \E_{z \in \F^*} | \widehat{\Delta_{zh} f_3}( zr) |^2\bigg) \cdot  \E_{z \in \F^*,h \in \F}  \Vert \Delta_{zh} f_3 \Vert_2^2   + O(p^{-1/2}).\end{align*}
However, for $z \in \F^*$ we have
\[ \E_{h \in \F} \Vert \Delta_{zh} f_3 \Vert_2^2 = \E_h \E_x |f_3(x+zh) \overline{f_3(x)}|^2 = \E_x|f_3(x)|^2\E_h{|f_3(x+zh)|^2} = \|f_3\|_2^4 \leq 1, \]
and hence
\begin{equation}\label{to-tim} |T(f_1, f_2, f_3, f_4)|^8 \leq \sup_{h \in \F^*,r \in \F } \E_{z \in \F^*} | \widehat{\Delta_{zh} f_3}( zr) |^2  + O(p^{-1/2}).\end{equation}
We are now almost ready to apply Lemma \ref{lem.tim} to (at last) bound this in terms of the $u_3^+$-norm of $f_3$. To do this, we need only extend the $z$-average to include $z = 0$. We have
\[ 
\E_{z \in \F^*} |\widehat{\Delta_{zh}f_3}(zr)|^2  \leq  \frac{p}{p-1}\E_{z \in \F} |\widehat{\Delta_{zh} f_3}(zr)|^2 \leq  \E_{z \in \F} |\widehat{\Delta_{zh}(f_3)}(zr)|^2 + \frac{2\Vert f_3 \Vert_{\infty}^4}{p}.
\]
Now we use $\|f_3\|_\infty \leq p^{1/16}$, \eqref{to-tim} and Lemma \ref{lem.tim} to obtain
\[ |T(f_1, f_2, f_3, f_4)|^8 \leq \sup_{h \in \F^*,r \in \F} \E_{z \in \F} | \widehat{\Delta_{zh} f_3}( zr) |^2  + O(p^{-1/2}) \leq \Vert f_3 \Vert_{u_3^+}^2 + O(p^{-1/2}),\] which is precisely the stated result. 
\end{proof}

Now we deduce Proposition \ref{gvn-nonlinear} itself.  A key ingredient in this process is the following decomposition result, reminiscent of results of ``Koopman von Neumann'' type. There are closely-related results in \cite{gowers-decomp, gowers-wolf1,gowers-wolf2,gowers-wolf3,trevisan-et-al}. However, in our particular setting it is not too hard to establish what we need quite directly. 
\begin{lemma}\label{decomposition-quadratic}
Suppose that $f : \F \rightarrow \C$ has $\|f\|_{2}\leq 1$ and that $\eps \geq 4p^{-1/8}$ is a parameter. Then there are complex numbers $\lambda_{\phi}$, $\phi \in Q(\F)$, such that
\[
\big\| f -  \sum_{\phi \in Q(\F)}{\lambda_{\phi} \phi} \big\|_{u^+_3} \leq \eps, \;\;  \big\|\sum_{\phi \in Q(\F)}{\lambda_{\phi} \phi}\big\|_{2}^2\leq 3 \;\; \mbox{and} \;\sum_{\phi \in Q(\F)} |\lambda_{\phi}| \leq \frac{4}{\eps}.
\]
\end{lemma}
\begin{proof}
Define
\begin{equation*}
\lambda_{\phi}:=\begin{cases} \langle f,{\phi}\rangle & \text{ if } |\langle f,{\phi}\rangle| \geq \eps/2\\
0 & \text{ otherwise}.\end{cases}
\end{equation*}
Write $I$ for the set of $\phi$ such that $\lambda_{\phi} \neq 0$, and let $I' \subset I$.  By Cauchy-Schwarz we have $|\lambda_{\phi}|\leq 1$ for all $\phi$.  Using this and the Gauss sum estimate
\[ |\langle \phi, \phi' \rangle | \leq \frac{1}{\sqrt{p}} \text{ whenever } \phi \neq \phi',\]we have
\begin{equation}\label{eq771}
\big \|\sum_{\phi \in I'}{\lambda_\phi\phi}\big\|_{2}^2 \leq \sum_{\phi \in I'}{|\lambda_\phi|^2} + \sum_{\phi\neq \phi'; \phi,\phi' \in I'}{|\langle \phi,{\phi'}\rangle|} \leq \sum_{\phi \in I'}{|\lambda_\phi|^2} + \frac{|I'|^2}{\sqrt{p}}.
\end{equation}
On the other hand from the definition of the $\lambda_\phi$s and the Cauchy-Schwarz inequality we have
\[ \sum_{\phi \in I'} |\lambda_{\phi}|^2 = \big\langle f, \sum_{\phi \in I'} \lambda_{\phi} \phi \big\rangle \leq \Vert f \Vert_2 \big\Vert \sum_{\phi \in I'} \lambda_{\phi} \phi \big\Vert_2 .\]
Comparing this with \eqref{eq771} (and using $\Vert f \Vert_2 \leq 1$) gives
\[ 
\big(\sum_{\phi \in I'}{|\lambda_\phi|^2}\big)^2 \leq \sum_{\phi \in I'}{|\lambda_\phi|^2} + \frac{|I'|^2}{\sqrt{p}};
\]
This implies that 
\begin{equation}\label{eq772} \sum_{\phi \in I'}{|\lambda_\phi|^2} \leq 1+ \frac{|I'|^2}{\sqrt{p}} .\end{equation}
This is true for all $I' \subset I$. Since $|\lambda_{\phi}| \geq \eps/2$ for all $\phi \in I$, it follows that 
\[ \frac{\eps^2 m}{4} \leq 1 + \frac{m^2}{\sqrt{p}} \text{ whenever } 0 \leq m \leq |I| \text{ and } m \in\N.\]
With the assumption that $\eps \geq 4 p^{-1/8}$ if follows that if $|I| \geq  8/\eps^2$ then we can take $m \in \N$ with $8/\eps^2\leq m < 8/\eps^2+1$ contradicting the above inequality.  Therefore 
\begin{equation}\label{ibound} |I| \leq \frac{8}{\eps^2} \leq p^{1/4}.\end{equation}

Taking $I' = I$ in \eqref{eq771} and  \eqref{eq772}, then using \eqref{ibound}, we have
\begin{equation}\label{eq773} \big \|\sum_{\phi \in I}{\lambda_\phi\phi}\big\|_{2}^2 \leq 1 + 2\frac{|I|^2}{\sqrt{p}} \leq 3.\end{equation}

Taking $I' = I$ in \eqref{eq772}, then using \eqref{ibound}, we have
\begin{equation}\label{eq774} \sum_{\phi \in Q(\F)} |\lambda_{\phi}| \leq \frac{2}{\eps} \sum_{\phi \in I} |\lambda_{\phi}|^2 \leq \frac{4}{\eps}.   \end{equation}

It follows from this and the Gauss sum estimate that if $\phi' \in I$ then
\begin{equation}\label{eq776}
\big|\langle f - \sum_\phi \lambda_\phi \phi , {\phi'}\rangle\big| = \big|\sum_{\phi \neq \phi'}{\lambda_\phi\langle \phi,\phi'\rangle}\big| \leq \frac{4}{\eps\sqrt{p}} \leq \eps,
\end{equation}
whilst if $\phi' \in Q(\F) \setminus I$ then
\begin{equation}\label{eq777}
\big|\langle f - \sum_\phi \lambda_\phi \phi , {\phi'}\rangle\big| = \big|\langle f,\phi'\rangle - \sum_{\phi}{\lambda_\phi\langle \phi,\phi'\rangle}\big| \leq \frac{\eps}{2} + \frac{4}{\eps\sqrt{p}} \leq \eps.
\end{equation}
Taken together, \eqref{eq773}, \eqref{eq774}, \eqref{eq776} and \eqref{eq777} cover all the statements we claimed, and the proof is complete.
\end{proof}

We are finally ready for the proof of Proposition \ref{gvn-nonlinear}, the main result of this section.
\begin{proof}[Proof of Proposition \ref{gvn-nonlinear}]
Set $\delta := |T(f_1,f_2,f_3,f_4)|$. If $\delta < 4 p^{-1/64}$ then the result is trivial, so suppose this is not the case. If $\inf_{i \in \{ 1,2,3,4\}} \Vert f_i \Vert_{\QM} = \Vert f_3 \Vert_{\QM}$ then the result follows immediately from Proposition \ref{gvn-3-l2}. It therefore suffices to show that under the assumptions just stated we have
\begin{equation}\label{to-show-final} \inf_{i \in\{ 1,2,4\}} \Vert f_i \Vert_{\QM} \gg \delta^5.\end{equation}

Apply Lemma \ref{decomposition-quadratic} with $f = f_3$ and with a parameter $\eps := 2^{-13}\delta^4$.  This gives us coefficients $(\lambda_\phi)_{\phi \in Q(\F)}$ such that
\[
\big\| f_3 -  \sum_{\phi \in Q(\F)}{\lambda_{\phi} \phi} \big\|_{u^+_3} \leq \eps, \;\; \big\|\sum_{\phi}{\lambda_{\phi} \phi}\big\|_{2}^2\leq 3 
\; \; \mbox{and} \; \sum_{\phi \in Q(\F)} |\lambda_{\phi}| \leq \frac{4}{\eps}.\]
Set
\begin{equation*}
g_3:=f_3 -  \sum_{\phi \in Q(\F)}{\lambda_{\phi} \phi};
\end{equation*}
thus $\|g_3\|_{2} \leq 1+\sqrt{3} < 4$, $\|g_3\|_{\infty} \leq 1+ \frac{4}{\eps} \leq \frac{8}{\eps}$ and $\|g_3\|_{u^+_3}\leq \eps$. By Proposition \ref{gvn-3-l2} we have, since $\eps \geq 8p^{-1/16}$,
\[ |T(f_1, f_2, g_3, f_4)|^8 \leq 2^{16}\big(\eps^2 + O(p^{-1/2})\big) \leq 2^{17}\eps^2,\]
provided $p$ is sufficiently large absolutely which we may certainly assume.  Thus
\[ \big | T(f_1, f_2, f_3, f_4) - \sum_{\phi} \lambda_{\phi} T(f_1, f_2, \phi, f_4) \big|^8 \leq 2^{17} \eps^2 < \left(\frac{\delta}{2}\right)^8.\]
It follows that 
\[ \big| \sum_{\phi} \lambda_{\phi} T(f_1, f_2,\phi, f_4) \big| \geq \frac{\delta}{2},\] and hence that there is some $\phi \in Q(\F)$ for which
\[ |T(f_1, f_2,\phi, f_4)| \geq \frac{\delta\eps}{8} \gg \delta^5.\]
Suppose that $\phi(t)=e_p(at^2+bt)$, and write $g_1(t):=f_1(t)e_p(at^2+bt)$, $g_2(t):=f_2(t)e_p(at^2+bt)$ and $g_4(t):=f_4(t)e_p(2at)$.  Then 
\[ T(g_1, g_2, 1_{\F}, g_4) = T(f_1, f_2,\phi, f_4),\] and hence
\[ |T(g_1, g_2, 1_{\F}, g_4)| \gg \delta^5.\]
By Proposition \ref{prop.mult} it follows that 
\[ \inf_{i \in\{ 1,2,4\}} \Vert g_i \Vert_{u_2^{\times}} + O\left(\frac{1}{p}\right) \gg \delta^5,\] and so
\[ \inf_{i \in\{ 1,2,4\}} \Vert g_i \Vert_{\QM} \gg \delta^5.\]
Since $\Vert f_i \Vert_{\QM} =\Vert g_i \Vert_{\QM}$, we have
\[ \inf_{i \in\{ 1,2,4\}} \Vert f_i \Vert_{\QM} \gg \delta^5.\]
This is precisely \eqref{to-show-final}, so the proof is concluded.
\end{proof}

To conclude this section we state Proposition \ref{gvn-nonlinear} in a qualitatively equivalent form, more useful for our later applications.

\begin{corollary}\label{gvn-cor}
There is a monotonically increasing function $\nu:(0,1] \rightarrow (0,1]$ with the following property. Suppose $h_1,h_2,h_3,h_4 : \F \rightarrow \C$ are such that $\|h_1\|_\infty,\|h_2\|_\infty,\|h_3\|_\infty,\|h_4\|_\infty \leq 1$ and $T(h_1,h_2,h_3,h_4) \geq \delta$. Then either $p \leq 1/\nu(\delta)$ or $\Vert h_i \Vert_{\QM} \geq \nu(\delta)$ for all $i \in \{ 1,2,3,4\}$.
\end{corollary}
In fact, $\nu(\delta)$ can be chosen to have the shape $\nu(\delta) \sim c \delta^C$.

\section[QM-systems and related concepts]{{$\QM$}-systems and related concepts}\label{sec3}

We begin by defining the notion of a $\QM$-system, an important concept in our paper. As is fairly standard, we write
 \[ S^1:=\{z \in \C: |z| = 1\}.\]
We shall also require a non-standard piece of notation: define 
\begin{equation*}
\G := \R/\Z \times \R/\Z \times S^1.
\end{equation*}
The group $\G$ is of course abelian, but it is convenient to use juxtaposition for the group operation; thus $(\theta_1, \theta_2, z) (\theta'_1, \theta'_2, z') = (\theta_1 + \theta'_1 , \theta_2 + \theta'_2, z z')$. We shall often be considering the product group $\G^d$, and we use the same convention there.
\begin{definition}
Let $d \in \N$. A $\QM$-system of dimension $d$ is a map $\Psi : \F \rightarrow \G^d$ of the form
\begin{equation*}
\Psi( x)=  (a_i x^2/p, 2a_i x/p, \psi_i(x))_{i = 1}^d,
\end{equation*}
where $a_i \in \F$ and $\psi_i \in \widehat{\F^*}$.  Recall from \S\ref{sec2} that we extend $\psi_i$ to all of $\F$ by setting $\psi_i(0) = 1$. 

Given a $\QM$-system $\Psi$ it will be helpful to write $\Psi' \supset \Psi$ to mean that $\Psi'$ extends $\Psi$ in the sense that $\Psi'$ has dimension $d' \geq d$ and $\Psi = \Pi \circ \Psi'$ where $\Pi:\G^{d'} \rightarrow \G^d$ is the projection onto the first $d$ co-ordinates.
\end{definition}

An important fact for us is that the ``orbit'' $\{ \Psi(x) : x \in \F\}$ is highly equidistributed inside a certain closed subgroup of $\G^d$. To explain what this group is, we need to make some definitions. Let $\Psi$ be a $\QM$-system as above. Then we associate to $\Psi$ the sublattices of $\Z^d$
\[ \Lambda_{\Psi}^+ := \{ \xi \in \Z^d : \xi_1 a_1 + \dots + \xi_d a_d \equiv 0 \pmod{p}\}\] and
\[ \Lambda_{\Psi}^{\times} := \{ \xi \in \Z^d: \psi_1^{\xi_1} \cdots \psi_d^{\xi_d} \; \mbox{is the trivial character}\}.\]
Note, incidentally, that $p \Z^d \subset \Lambda_{\Psi}^+$ and $(p-1)\Z^d \subset \Lambda_{\Psi}^{\times}$, and so both $\Lambda_{\Psi}^+$ and $\Lambda_{\Psi}^{\times}$ have full rank.
With these lattices defined we introduce the closed subgroups
\[ G_{\Psi}^+ := \{ g \in (\R/\Z)^d : \xi \cdot g = 0 \; \mbox{for all $\xi \in \Lambda_{\Psi}^{+}$}\}\]
and 
\[ G_{\Psi}^{\times} := \{ z \in (S^1)^d : z^{\xi} = 1 \; \mbox{for all $\xi \in \Lambda_{\Psi}^{\times}$}\},\] where here $\xi \cdot g := \xi_1 g_1 + \dots + \xi_d g_d$ and $z^{\xi} := z_1^{\xi_1} \cdots z_d^{\xi_d}$. These two groups $G_{\Psi}^+$ and $G_{\Psi}^{\times}$ are both closed subgroups of compact groups ($(\R/\Z)^d$ and $(S^1)^d$ respectively) and as such they carry natural Haar probability measures $\mu_{G_{\Psi}^+}$ and $\mu_{G_{\Psi}^{\times}}$. Define
\[ H_{\Psi} := G_{\Psi}^+ \times G_{\Psi}^+ \times G_{\Psi}^{\times},\] and put the natural probability measure
\[ \mu_{H_{\Psi}} := \mu_{G_{\Psi}^+} \times \mu_{G_{\Psi}^+} \times \mu_{G_{\Psi}^{\times}}\] on this group. By abuse of notation, we regard this as a probability measure on $\G^d$ as well (which is permissible, as $H_{\Psi}$ is a subgroup of $\G^d$). Note that $\Psi(\F) \subset H_{\Psi}$. It turns out that $\Psi(\F)$ is close to being equidistributed in $H_{\Psi}$. To formulate this fact in a convenient form (Proposition \ref{distribution} below), we need a further definition.

\begin{definition}[Trig norm]\label{deftrig}
Let $F : \G^d \rightarrow \C$ be a function. We write $\Vert F \Vert_{\trig}$ for the smallest $M \in [0,\infty]$ such that $F$ has a Fourier expansion 
\[ F(\theta_1, \theta_2, z) = \sum_{\|\xi_1\|_1, \|\xi_2\|_1, \|\xi_3\|_1 \leq M} \hat{F}(\xi_1, \xi_2,\xi_3) e(\xi_1 \cdot \theta_1 + \xi_2 \cdot \theta_2) z^{\xi_3}\] with $\sum_{\xi_1,\xi_2,\xi_3} |\hat{F}(\xi_1, \xi_2,\xi_3)| \leq M$.  
\end{definition}

We call this a norm, although ``measure of complexity'' would be more accurate. We have 
\begin{equation} \label{subadd} \Vert F_1 + F_2 \Vert_{\trig} \leq \Vert F_1 \Vert_{\trig} + \Vert F_2 \Vert_{\trig},\end{equation}
\begin{equation*}  \Vert F_1F_2 \Vert_{\trig} \leq \max\left\{\Vert F_1 \Vert_{\trig} + \Vert F_2 \Vert_{\trig},\|F_1\|_{\trig}\|F_2\|_{\trig}\right\},\end{equation*}
and
\begin{equation}\label{scale}\Vert \lambda F \Vert_{\trig} \leq \max(1, |\lambda|) \Vert F \Vert_{\trig} \text{ for all } \lambda \in \C,\end{equation}
as well as the shift property that if we define $T_hF(g) := F(h^{-1}g)$ then 
\begin{equation*}
\Vert T_h F \Vert_{\trig} = \Vert F \Vert_{\trig} \text{ for all }h \in \G^d.
\end{equation*}
We call those functions $F : \G^d \rightarrow \C$ with finite trig norm \emph{trigonometric polynomials}. They are dense in $C(\G^d)$ (with the sup norm). This follows from the Stone-Weierstrass theorem: the trigonometric polynomials form an algebra and contain the characters, hence separate points. (This could also be established directly using harmonic analysis.)

We turn now to the promised equidistribution statement. We call it the ``baby counting lemma'' as it is a simpler cousin of one of the key ingredients of our work which we shall call the counting lemma.

\begin{proposition}[Baby counting lemma]\label{distribution}
Let $\Psi$ be a $\QM$-system of dimension $d$, and let $F : \G^d \rightarrow \C$ be a trigonometric polynomial. Then we have
\[ \E_{x \in \F} F(\Psi(x)) = \int F d\mu_{H_{\Psi}} + o_{p \rightarrow \infty}(\Vert F \Vert_{\trig}).\]
\end{proposition}

The proof of Proposition \ref{distribution} relies on estimates for character sums twisted by quadratic phases. We recall what we need on this topic at the beginning of \S\ref{sec8}, and give the proof of Proposition \ref{distribution} later in that same section. 

It is useful to have some ``absolute values'' associated to the groups we are interested in.  For $\theta \in \R/\Z$ we write
\begin{equation}\label{eqn.nj}
| \theta |:=\|\theta\|_{\R/\Z}:=\min\{|\delta| : \delta \in \R, \theta + \delta \in \Z\}.
\end{equation}
Since $S^1$ already has a notion of absolute value inherited from $\C$, we have to be slightly careful and for $z \in S^1$ write
\begin{equation*}
\|z\|_{S^1}:=\left\|\frac{1}{2\pi i}\log z\right\|_{\R/\Z}:=\left| \frac{1}{2\pi i}\log z\right|,
\end{equation*}
where $\frac{1}{2\pi i}\log z$ naturally takes values in $\R/\Z$ and so we can use the notation in (\ref{eqn.nj}).  These combine in the obvious way for $\G$ so that
\[ |(\theta_1,\theta_2,z)| :=  \max \{|\theta_1|,|\theta_2|,\|z\|_{S^1}\}.\]
Finally for $(\R/\Z)^d$ and $\G^d$ we put
\begin{equation*}
| (g_1,\dots, g_d) | := \max_i |g_i|.
\end{equation*}
Thus $|\cdot|$ is used in several different ways in our paper and in any given situation its meaning must be inferred from context. This should not be difficult.  

The last piece of notation we need is for boxes in $\G^d$: given $\eps > 0$ define
\begin{equation}
\label{eqn.box}
X(\eps) := \{x \in \G^d : |x| \leq \eps\}.
\end{equation}
The group $H_{\Psi}$ may look rather complicated with respect to $|\cdot |$ (on $\G^d$).  (Informally, it may ``wind around'' $\G^d$ a very large number of times). However, we do at least have some control, as shown in Lemma \ref{box-pigeon} below. We deduce that lemma from the following more general fact which will be useful in \S\ref{sec7}.
\begin{lemma}\label{pigeon-projection}
Suppose that $G$ is a compact abelian group with Haar measure $\mu_G$, and that $\pi : G \rightarrow (\R/\Z)^d$ is a homomorphism. Then $\mu_G (\{ x : |\pi(x)| \leq \delta\}) \geq \delta^d$.
\end{lemma}
\begin{proof}
For $\theta \in (\R/\Z)^d$, write $B_{\delta}(\theta) := \prod_{i = 1}^d [\theta_i, \theta_i + \delta]$. For any fixed $x \in (\R/\Z)^d$ we have
\[ \int 1_{B_{\delta}(\theta)}(x) d\theta = \delta^d.\]
Taking $x = \pi(g)$ and integrating over $G$, we obtain
\[ \int 1_{B_{\delta}(\theta)} (\pi(g))d\mu_G(g) d\theta = \delta^d.\]
Hence there is some $\theta$ such that 
\[ \mu_G (\{ g : \pi(g) \in B_{\delta}(\theta)\}) \geq \delta^d.\]
Write $S := \{ g : \pi(g) \in B_{\delta}(\theta)\}$, thus $\mu_G(S) \geq \delta^d$.
Note that if $g_1, g_2 \in S$ then $|\pi(g_1 - g_2) | \leq \delta$. Thus the result follows from the fact that $\mu_G(S - S) \geq \mu_G(S - s) = \mu_G(S)$, where $s \in S$ is any element. 
\end{proof}
Note that this proof is really just the same as that of \cite[Lemma 4.19]{taovu}.

\begin{lemma}\label{box-pigeon}
We have $\mu_{H_{\Psi}}(X(\eps)) \geq \eps^{3d}$.
\end{lemma}
\begin{proof}
Apply the previous lemma with $\pi$ being the restriction to $H_{\Psi}$ of the natural homomorphism from $\G^d$ to $(\R/\Z)^{3d}$.
\end{proof}

A useful corollary of Proposition \ref{distribution} and Lemma \ref{box-pigeon} is the following assertion.

\begin{corollary}\label{cor3.3}
There is a function $p_1:\Z_{\geq 0} \times (0,1] \rightarrow \N$ with $p_1(d',\eps') \geq p_1(d,\eps)$ whenever $\eps' \leq \eps$ and $d' \geq d$, such that if $p \geq p_1(d,\eps)$ then the following is true.

For every $d$-dimensional $\QM$-system $\Psi$ and $h \in H_{\Psi}$ we have
\[ \mu_{\F} (\{x : |h^{-1}\Psi(x)| \leq \eps\})  \geq \frac{1}{8}\left(\frac{\eps}{4}\right)^{3d}.\]
\end{corollary}
\begin{proof}
We use the fact that the trigonometric polynomials are dense in $C(\G^d)$. Thus if $\eps > 0$ then there is some $F_0: \G^d \rightarrow \R$ such that 
\begin{itemize}
\item $F_0 \geq -\eps'$ pointwise, where $\eps' := \frac{1}{2}\left(\frac{\eps}{2}\right)^{3d}$;
\item $F_0 \leq 0$ outside $X(\eps)$, the box defined in (\ref{eqn.box});
\item $F_0 \geq 1$ on $X(\eps/2)$;
\item $F_0 \leq 2$ on $\G^d$ and
\item $\Vert F_0 \Vert_{\trig} = O_{\eps, d}(1)$.
\end{itemize}
Now apply Proposition \ref{distribution} with $F = T_hF_0$. We have $\Vert T_h F_0 \Vert_{\trig} = \Vert F_0 \Vert_{\trig} = O_{\eps, d}(1)$, uniformly in $h$. By Proposition \ref{distribution},
\[ \E_{x \in \F} T_hF_0(\Psi(x)) = \int T_h F_0 d\mu_{H_{\Psi}} + o_{p \rightarrow \infty}(\Vert F_0 \Vert_{\trig} ) = \int F_0 d\mu_{H_{\Psi}} + o_{\eps, d;p\rightarrow \infty}(1),\] the second equality being a consequence of the invariance of Haar measure.  Since $F_0 \geq 1$ on $X(\eps/2)$ and $F_0 \geq -\eps'$ everywhere, it follows from Lemma \ref{box-pigeon} that 
\[ \int F_0 d\mu_{H_{\Psi}} \geq \left(\frac{\eps}{2}\right)^{3d} - \eps' = \frac{1}{2}\left(\frac{\eps}{2}\right)^{3d}  .\] Thus 
\[ \E_{x \in \F} T_hF_0(\Psi(x)) \geq \frac{1}{2}\left(\frac{\eps}{2}\right)^{3d} - o_{\eps,d;p \rightarrow \infty}(1) \geq \frac{1}{4}\left(\frac{\eps}{2}\right)^{3d}\]
provided $p \geq p_0(d,\eps)$. However, $T_h F_0(y) \leq 0$ unless $h^{-1} y \in X(\eps)$ and $F_0 \leq 2$ pointwise, so $T_hF_0(y) \leq 2 \cdot 1_{X(\eps)}(h^{-1}y)$.  It follows that
\begin{equation*}
\mu_{\F} (\{x : |h^{-1}\Psi(x)| \leq \eps\})  \geq \frac{1}{8}\left(\frac{\eps}{2}\right)^{3d}
\end{equation*}
provided $p \geq p_0(d,\eps)$.  It remains to set
\begin{equation*}
p_1(d,\eps):=\max\left\{ p_0(d',2^{-j}): 2^{-j} \geq \eps/2, d' \leq d,\text{ and } d,j \in \Z_{\geq 0}\right\}.
\end{equation*}
This function $p_1$ is monotonic in the desired sense, and the result follows since there is some $j \in \Z_{\geq 0}$ with $\eps \geq 2^{-j} \geq \eps/2$ and
\begin{equation*}
\mu_{\F} (\{x : |h^{-1}\Psi(x)| \leq \eps\}) \geq \mu_{\F} (\{x : |h^{-1}\Psi(x)| \leq 2^{-j}\})  \geq \frac{1}{8}\left(\frac{\eps}{4}\right)^{3d}
\end{equation*}
whenever $p \geq p_1(d,\eps)$.
\end{proof}

\emph{Monotonic functions.}  By invoking the Stone-Weierstrass theorem we have taken a soft approach to approximating intervals by trigonometric polynomials in Corollary \ref{cor3.3}. This adds a slight technical complexity because the constant behind the $O_{\eps,d}(1)$ term in the proof could, in principle, depend in a very peculiar way on $\eps$ and $d$.  This complexity will crop up in other parts of the argument and we shall deal with it by ensuring that ``universal functions'' (in the above case $p_1$) are monotonic in a suitable sense.

We take the following convention: all our ``universal functions'' will be of the form $f:D_1\times \dots \times D_r \rightarrow D_0$ where the $D_i$s are one of the sets $(0,1],\Z_{\geq 0}, \N,\R_{\geq 0}$ or $\R_{\geq 1}$.  If $D_i=(0,1]$ then we write $x \preccurlyeq_i y$ whenever $x \leq y$ ; otherwise $x \preccurlyeq_i y$ whenever $x \geq y$.  We shall say that $f$ is \emph{monotone} if $f(x)  \preccurlyeq_0 f(y)$ whenever $x_i \preccurlyeq_i y_i$ for all $1 \leq i \leq r$.

The two places we have encountered monotonicity so far are in Corollary \ref{cor3.3} where $p_1$ is monotonic in the above sense and Corollary \ref{gvn-cor} where $\nu$ is monotonic.

Finally we come to a crucial definition in the paper.

\begin{definition}[$\QM$-Bohr set]
Suppose that $\Psi$ is a $\QM$-system of dimension $d$. Then we write $B(\Psi, \eps) := \Psi^{-1}(X(\eps))$, and call this a $\QM$-Bohr set of dimension $d$ and width $\eps$. That is, $B(\Psi, \eps) := \{x \in \F : |\Psi(x)| \leq \eps\}$. 
\end{definition}
It should not come as a surprise that we have an analogue of the usual lower bound for the size of Bohr sets \cite[Lemma 4.19]{taovu}.
\begin{lemma}\label{bohr-lower}
There is a monotonic function $\beta : \Z_{\geq 0} \times (0,1] \rightarrow (0,1]$ such that the following is true. For all $d$-dimensional $\QM$-systems $\Psi$ and parameters $\eps \in (0,1]$ we have
\[ \mu_\F(B(\Psi, \eps)) \geq \beta(d,\eps).\]
\end{lemma}
\begin{proof}
This is immediate from Corollary \ref{cor3.3} by taking $h = \id_{\G^d} = ((0,0,1), \dots , (0,0,1))$, the identity element in $\G^d$, and then putting
\begin{equation*}
\beta(d,\eps):=\min\left\{ \frac{1}{p_1(d,\eps)}, \frac{1}{8}\left(\frac{\eps}{4}\right)^{3d}\right\}.
\end{equation*}
Since $\id_{\G^d} \in B_{\QM}(\Psi, \eps)$ for all $\eps \in (0,1]$ the result follows.
\end{proof}

\section{Structure of the main argument }\label{sec4}

We turn now to a discussion of the basic form of the rest of the argument. The basic strategy is to work by induction on the number of colours, but to get this to work effectively we must establish a more general statement.  It may be useful to recall the notion of monotone function we are using from the discussion after Corollary \ref{cor3.3}.

\begin{proposition}\label{main-prop-refined}
There are monotonic functions $\eta,\zeta : (0,1] \times \Z_{\geq 0} \times \Z_{\geq 0} \rightarrow (0,1]$ and $p_0 : (0,1] \times \Z_{\geq 0} \times \Z_{\geq 0} \rightarrow \N$ with the following property. Suppose that $B \subset \F$ is a $\QM$-Bohr set of dimension $d$ and width $\delta$, and that we have an $r$-colouring $c: \tilde B \rightarrow [r]$ for some set $\tilde B \subset B$ with $|\tilde B| \geq (1 - \eta(\delta,d,r))|B|$. Then, provided $p\geq p_0(\delta,d,r)$ there are $\zeta(\delta,d,r)p^2$ pairs $(x, y)$ for which $x,y, x+y, xy \in \tilde B$ and $c(x) = c(y) = c(x+y) = c(xy)$.
\end{proposition}

\emph{Remarks.} This immediately gives our main result. One may think of $c$ as an ``almost'' colouring (or more precisely a $(1 - \eta(\delta,d,r))$-almost $r$-colouring). The nature of our inductive arguments necessitates the consideration of almost colourings in addition to true colourings.

\begin{proof}[Proof of Theorem \ref{mainthm}]
Apply Proposition \ref{main-prop-refined} with $d=0$, $\delta=1/2$ and $\tilde B = B = \F$.  Since $(0,0,0+0,0 \cdot 0)$ is always monochromatic we see that every colouring contains at least $1$ monochromatic quadruple.  It follows that the total number of monochromatic quadruples is at least
\begin{equation*}
\min\big\{\frac{1}{p_0(\frac{1}{2},0,r)^2}, \zeta(\frac{1}{2},0,r)\big\}p^2 \gg_r p^2.
\end{equation*}
The result is proved.
\end{proof}

The proof of Proposition \ref{main-prop-refined} combines three fairly substantial pieces: a regularity lemma, a counting lemma, and a Ramsey-theory result. These three ingredients and their proofs can be understood more-or-less independently, the only common features being the language of \S\ref{sec3}. \vspace*{8pt}

\emph{Regularity lemma.} The basic principle of the regularity lemma is that it allows one to replace an arbitrary colouring of $\F$ by one that is induced from a ``nice'' colouring of $\G^d = \big( \R/\Z \times \R/\Z \times S^1 \big)^d$ by pullback under a $\QM$-map $\Psi : \F \rightarrow \G^d$.

This is a well-trodden idea which of course goes back to Szemer{\'e}di \cite{sze}.  Perhaps closer to our particular instance is the arithmetic development in \cite[Theorem 5.2]{green-reg}; the probabilistic framing in \cite[Theorem 2.11]{tao}; and the combination in \cite[Theorem 1.2]{green-tao-regularity}.

\begin{proposition}[Regularity lemma]\label{regularity}
Suppose that $\Omega : \Z_{\geq 0} \times \Z_{\geq 0} \times (0,1] \times \R_{\geq 0} \times \Z_{\geq 0} \rightarrow \R_{\geq 1}$ is monotonic.  Then there are monotonic functions $M_\Omega,D_\Omega$, and $p_\Omega$ mapping $\Z_{\geq 0} \times \Z_{\geq 0} \times (0,1] \times (0,1]$ to $\R_{\geq 1}$, $\Z_{\geq 0}$, and $\N$ respectively such that the following holds.

Suppose that $\Psi$ is a $d$-dimensional $\QM$-system of width $\delta$, $B:=B(\Psi,\delta)$, $\tilde B \subset B$ has $|\tilde B| \geq (1 - \frac{\eps^2}{100})|B|$ for some parameter $\eps \in (0,1]$, and $c : \tilde B \rightarrow [r]$ is an $r$-colouring of $\tilde B$.  Then there is a $\QM$-system $\Psi' \supset \Psi$ of some dimension $d'$, functions $F_1,\dots, F_r : \G^{d'} \rightarrow \R_{\geq 0}$, and functions $g_1,\dots, g_r : \F \rightarrow [-1,1]$, such that
\begin{enumerate}
\item $\Vert F_i \Vert_{\trig} \leq M_{\Omega}(r,d,\delta,\eps)$ for all $i \in\{ 1,\dots, r\}$;
\item $d' \leq D_{\Omega}(r,d,\delta,\eps)$;
\item $\Vert F_i \circ \Psi' - g_i  \Vert_2 \leq \eps$ for all $i \in\{ 1,\dots, r\}$;
\item $\Vert 1_{c^{-1}(i)} - g_i\Vert_{\QM} \leq 1/\Omega(r,d,\delta,\Vert F_i \Vert_{\trig}, d')$;
\item $\sum_{i = 1}^r F_i \circ \Psi' \geq 1$ pointwise on $B(\Psi, \frac{\delta}{2})$;
\item $\|F_i \circ \Psi'\|_{4} \leq 2$ for all $i \in \{1,\dots,r\}$;
\end{enumerate}
provided $p \geq p_\Omega(r,d,\delta,\eps)$.
\end{proposition}

The proof of the regularity lemma involves an ``energy increment'' argument of a type familiar to experts, but some of the technical details are a little tricky. The proof is given in \S \ref{sec5}.\vspace*{8pt}

\emph{Counting lemma.}  As usual we complement the regularity lemma with a counting lemma.  There is always a trade-off between the effort needed to prove a regularity lemma and its companion counting lemma; in our work the former contains essentially all of the difficulties.

The basic idea of the counting lemma is that if we have a colouring on $\F$ pulled back from a colouring of $\G^d$ under a $\QM$-map (as provided by the regularity lemma) then the number of monochromatic quadruples $x,y,x+y, xy$ in $\F$ is related to the number of monochromatic copies of a certain type of \emph{linear} configuration in $\G^d$: specifically, if $y$ is constrained to lie in a small $\QM$-Bohr set then $\Psi(x), \Psi(x+y), \Psi(xy)$ correspond to triples of the form $(t,u,v), (t + u', u, v'), (t', u', v)$. To get an idea of why this is so consider the case when $d=1$. In this case we can put
\begin{equation*}
\Psi:\F \rightarrow \G;x \mapsto (ax^2/p, 2ax/p,\chi(x)),
\end{equation*}
and the colouring of $\F$ is pulled back from a colouring of $\G$.  If $\Psi(y) \approx \id_{\G} = (0,0,1)$ then the remaining three points $\Psi(x)$, $\Psi(x+y)$, and $\Psi(xy)$ have constraints resulting from the identities
\begin{equation*}
x^2 + 2xy + y^2 = (x + y)^2,\;  x + y = (x + y) \; \text{ and } \; \chi(x)\chi(y)\;  \mbox{``$=$''}\; \chi(xy)
\end{equation*}
(where the quotation marks in the last are because our characters $\chi$ are not $0$ at $0$),
which imply that 
\[ \frac{ax^2}{p} + \frac{2axy}{p} \approx \frac{a(x + y)^2}{p},\;\; \frac{2ax}{p} \approx \frac{2a(x+y)}{p} \; \text{and} \; \chi(x) \approx \chi(xy).\]
  This tells us that $\Psi(x), \Psi(x+y), \Psi(xy)$ have approximately the form $(t,u,v)$, $(t + u',u,v')$, and $(t', u',v)$ respectively.  The counting lemma is a much stronger statement than this, essentially asserting that the above is the \emph{only} type of constraint that occurs.

\begin{proposition}[Counting lemma]\label{counting-lem}
Suppose $\Psi$ is a $d$-dimensional $\QM$-system, $F : \G^d \rightarrow \C$ is a trigonometric polynomial, and that $S \subset B(\Psi,\eps)$.  Then \begin{align*} & T(F \circ \Psi, 1_S, F \circ \Psi, F \circ \Psi) \\ & \qquad =   \mu_{\F}(S) \int F(t,u,v) F(t + u', u, v') F(t', u', v) d\mu_{H_{\Psi}}(t,u,v) d\mu_{H_{\Psi}}(t',u',v') \\ & \qquad  \qquad +   O(\epsilon \mu_{\F}(S)\Vert F \Vert_{\trig}^4) + o_{p \rightarrow \infty}(\Vert F \Vert_{\trig}^{9d}).
\end{align*}
\end{proposition}

The counting lemma is established using harmonic analysis in \S \ref{sec8}. It relies on bounds for certain character sums (mixing quadratic phases and shifts of multiplicative characters), which in turn use some fairly deep number-theoretic inputs. These issues are discussed at the beginning of \S \ref{sec8}.\vspace*{8pt}

\emph{Ramsey lemma.} Finally, we need a result of a Ramsey-theoretic nature. The counting lemma above transfers our nonlinear question about quadruples $x, y, x+y,xy$ to a question about linear configurations, but we must still solve this linear problem. A toy version of the result we need is the following: if $G$ is a sufficiently large abelian group and if $c : G \times G \times G \rightarrow [r]$ is an $r$-colouring, we may find $t,t',u,u',v,v'$ with $c(t,u,v) = c(t + u', u, v') = c(t',u',v)$. We do indeed prove such a statement, but again we need something a little more complicated for the purposes at hand, designed to dovetail with the conclusions of the counting lemma.

\begin{proposition}[Ramsey lemma]\label{ramsey-prop} There is a monotonic function $\rho : \Z_{\geq 0} \times \Z_{\geq 0} \times (0,1] \rightarrow (0,1]$ with the following property. Suppose that $X$ and $Y$ are compact Abelian groups with Haar probability measures $\mu_X, \mu_Y$, that $\pi_X : X \rightarrow (\R/\Z)^{d}$ and $\pi_Y : Y \rightarrow (\R/\Z)^{d}$ are continuous homomorphisms, and that $F_1,\dots, F_r : X \times X \times Y \rightarrow \R_{\geq 0}$ are continuous functions with $\sum_{i=1}^r F_i(x_1, x_2, y) \geq 1$ whenever we have $| \pi_X(x_1)|,|\pi_X(x_2)|,|\pi_Y(y)| \leq \delta/4$. Let $\mu = \mu_X \times \mu_X \times \mu_Y$. 
Then
\[ {\int F_i(t,u,v) F_i(t + u', u, v') F_i(t', u',v)} d\mu(t,u,v) d\mu(t',u',v') \geq \rho(r,d,\delta)
\] for some $i \in [r]$.
\end{proposition}

This result is established in \S \ref{sec7}. We again proceed by induction on the number of colours (thus, taken as a whole, our paper has two nested inductions on the number of colours). The basic scheme of the argument is inspired by work of Cwalina and Schoen \cite{cwalina-schoen} on Rado's theorem, but the details are quite different. We make critical use of the ``dependent random choice'' technique pioneered by Gowers \cite{gowers-4ap}.\vspace*{8pt} 

\emph{Proposition \ref{main-prop-refined}.} We shall shortly give the proof of Proposition \ref{main-prop-refined}, assuming the regularity, counting and Ramsey lemmas. First we note that the counting and Ramsey lemmas may be combined to give the following statement.

\begin{proposition}\label{prop99}
There are monotonic functions $\kappa : \Z_{\geq 0} \times \Z_{\geq 0}  \times (0,1] \times \R_{\geq 0} \rightarrow (0,1]$ and $p_2:\R_{\geq 0} \times \Z_{\geq 0}\times \Z_{\geq 0}\times (0,1] \times (0,1]\rightarrow \N$ with the following property. Suppose that $d' \geq d$ are integers, $M \geq 1$ is real, $\sigma \in (0,1]$; that $\Psi' \supset \Psi$ are $\QM$-systems of dimension $d',d$, and that $F_1,\dots, F_r : \G^{d'} \rightarrow \R_{\geq 0}$ are functions with $\Vert F_i \Vert_{\trig} \leq M$ and $\sum_{i = 1}^r F_i \circ \Psi' \geq 1$ pointwise on $B(\Psi, \frac{1}{2}\delta)$. Then for any sets $S_1,\dots,S_r$ with $S_i \subset B(\Psi',\kappa(r,d,\delta,\|F_i\|_{\trig}))$ and $\mu_{\F}(S_i) \geq \sigma$ there is some $i$ such that
\[T(F_i \circ \Psi', 1_{S_i}, F_i \circ \Psi', F_i \circ \Psi') \geq \textstyle\frac{1}{16}\displaystyle\mu_{\F}(S_i)\rho(r,d,\delta) \] provided that $p \geq p_2(M,d',r,\sigma,\delta)$. (Where $\rho$ is as in the Ramsey lemma, Proposition \ref{ramsey-prop}.)
\end{proposition}
\emph{Remark.} A crucial point here is that neither $\kappa$ nor $\rho$ depends on $d'$. If they did, our arguments would be circular.
\begin{proof}
Apply the counting lemma (Proposition \ref{counting-lem}) to each of the $F_i$s and the $\QM$-system $\Psi'$ to get that
\begin{align*} & T(F_i \circ \Psi', 1_{S_i}, F_i \circ \Psi', F_i \circ \Psi') \\ & \qquad =   \mu_{\F}(S) \int F_i(t,u,v) F_i(t + u', u, v') F_i(t', u', v) d\mu_{H_{\Psi'}}(t,u,v) d\mu_{H_{\Psi'}}(t',u',v') \\ & \qquad  \qquad +   O(\kappa \mu_{\F}(S)\|F_i\|_{\trig}^4) + o_{p \rightarrow \infty}(M^{9d'})\\
& \qquad \geq \mu_{\F}(S) \int F_i(t,u,v) F_i(t + u', u, v') F_i(t', u', v) d\mu_{H_{\Psi'}}(t,u,v) d\mu_{H_{\Psi'}}(t',u',v')\\
& \qquad \qquad - \textstyle\frac{1}{16}\displaystyle\mu_{\F}(S_i)\rho(r,d,\delta)
\end{align*}
provided $p \geq p_3(M,r,d',\delta,\sigma)$ (where $p_3$ can be taken monotone increasing in its arguments) and $\kappa=c\rho(r,d,\delta)\min ( 1, \|F_i\|_{\trig}^{-4})$ for some absolute $c>0$, which satisfies the relevant monotonicity properties since $\rho$ does. 

Let $\eps:=\delta/12 \pi rd'(1+M)^2$; the reason for this choice will become clear later.  We should like to apply the Ramsey lemma to which end we put $X:=G_{\Psi'}^+$ and $Y:=G_{\Psi'}^\times$ then $H_{\Psi'} = X \times X \times Y$.  Moreover, if we write $\pi_X:G_{\Psi'}^+ \rightarrow G_{\Psi}^+$, and $\pi_Y:\G_{\Psi'}^\times \rightarrow G_{\Psi}^\times$ for the respective projections onto the first $d$-coordinates then these are well-defined continuous homomorphisms, as is $\pi:H_{\Psi'} \rightarrow H_{\Psi}$ defined by $\pi(x_1,x_2,y) = (\pi_X(x_1),\pi_X(x_2),\pi_Y(y))$.

If $h=(x_1,x_2,y) \in H_{\Psi'}$, then by Corollary \ref{cor3.3} (provided $p \geq p_1(d',\eps)$) there is some $z \in \F$ with $|h^{-1}\Psi'(z)| \leq \eps$.   

Two things follow from this. First, $|\pi(h)^{-1}\pi(\Psi'(z))| \leq \eps$, so if $|\pi(h)| \leq \delta/4$ then $|\Psi(z)| \leq \frac{1}{4}\delta + \eps \leq \frac{1}{2}\delta$ by the triangle inequality \emph{i.e.} $z \in B(\Psi,\frac{1}{2}\delta)$.

Secondly, it is easy to see that the functions $F_i$ are Lipschitz in the sense that
\begin{equation*}
|F_i(w)-F_i(v)| \leq 6\pi d'M^2|w^{-1}v| \text{ for all } w,v \in \G^{d'},
\end{equation*}
so they are continuous, but we also have
\begin{equation*}
\big|\sum_{i}{F_i(h)} - \sum_i{F_i(\Psi'(z))}\big| \leq 6 \pi rd'M^2|h^{-1}\Psi'(z)| \leq 6\pi \eps rd'M^2 \leq \textstyle\frac{1}{2}\displaystyle.
\end{equation*}

We conclude from these two facts and the hypothesis that $\sum_i{F_i\circ \Psi'} \geq 1$ on $B(\Psi,\frac{1}{2}\delta)$, that if $h  \in H_{\Psi'}$ has $|\pi(h)| \leq \frac{1}{4}\delta$ then
\begin{equation*}
\sum_{i}{F_i(h)} \geq  \sum_i{F_i(\Psi'(z))}- \textstyle\frac{1}{2}\displaystyle \geq \textstyle\frac{1}{2}\displaystyle.
\end{equation*}
We now apply the Ramsey lemma (Proposition \ref{ramsey-prop}) to the functions $2F_i$ to get that
\begin{equation*}
\int F_i(t,u,v) F_i(t + u', u, v') F_i(t', u', v) d\mu_{H_{\Psi'}}(t,u,v) d\mu_{H_{\Psi'}}(t',u',v') \geq \textstyle\frac{1}{8}\displaystyle\rho(r,d,\delta)
\end{equation*}
for some $i \in [r]$.  This gives the result provided
\begin{equation*}
p \geq p_2(M,d',r,\sigma,\delta):=p_3(M,r,d',\delta,\sigma)+p_1(d',\delta/12 \pi rd'M^2)
\end{equation*}
which is easily seen to be monotone.  The proposition is proved.
\end{proof}

Finally, we record a fairly simple application of the Cauchy-Schwarz inequality. 

\begin{lemma}\label{simple-lem}
Suppose that $S \subset \F$ and $f_1,f_3,f_4:\F \rightarrow [-K,K]$ are functions with $\|f_1\|_4,\|f_3\|_4,\|f_4\|_4 \leq 3$. Then 
\[ |T(f_1, 1_S, f_3, f_4)| \leq \frac{K^3}{p} + 9\mu_{\F}(S)\min_i \Vert f_i \Vert_{2}.\]
\end{lemma}
\begin{proof}
By the Cauchy-Schwarz inequality we have
\begin{align*}
|T(f_1, 1_S, f_3, f_4)| & = |\E_{x,y} f_1(x)1_S(y) f_3(x+y) f_4(xy)| \\ & \leq \frac{K^3}{p} + |\E_{x,y} 1_{\F^*}(y)f_1(x)1_S(y) f_3(x+y) f_4(xy)|\\
& \leq \frac{K^3}{p} + \mu_{\F}(S) \sup_{y \in \F^*} \E_x |f_1(x) f_3(x+y) f_4(xy)|   \\ & \leq \frac{K^3}{p} + \mu_{\F}(S) \sup_{y \in \F^*} (\E_x f_1(x)^2)^{1/2} (\E_x f_3(x+y)^2 f_4(xy)^2 )^{1/2}\\  & \leq \frac{K^3}{p} + \mu_{\F}(S) \sup_{y \in \F^*} (\E_x f_1(x)^2)^{1/2} (\E_x f_3(x+y)^4)^{1/4}(\E_xf_4(xy)^4 )^{1/4}\\ &  \leq \frac{K^3}{p} + 9\mu_{\F}(S) \Vert f_1 \Vert_{2}.
\end{align*}
The proofs for $i \in \{ 3,4\}$ are very similar. 
\end{proof}
\newcommand\gvn{\operatorname{gvn}}
We are now ready for the proof of Proposition \ref{main-prop-refined}.
\begin{proof}[Proof of Proposition \ref{main-prop-refined}]
We proceed by induction on $r$, setting $\eta( \cdot, \cdot , 0), \zeta(\cdot,\cdot,0) = \frac{1}{2}$ and $p_0(\cdot,\cdot,0) =1$. All of these functions are monotone and vacuously satisfy the proposition since there is no $0$-colouring of a non-empty set, and $|\tilde B| \geq \frac{1}{2}|B| >0$ since every $\QM$-Bohr set contains $0$.

Suppose we have established the existence of $\eta( \cdot, \cdot, r-1),\zeta(\cdot,\cdot,r-1)$ and $p_0(\cdot,\cdot,r-1)$ satisfying the desired conclusion; we shall show how to define $\eta( \cdot, \cdot, r),\zeta(\cdot,\cdot,r)$ and $p_0(\cdot,\cdot,r-1)$.  The argument we present is a little delicate and so we shall be quite explicit.  We shall make use of the following functions:
\begin{enumerate}
\renewcommand{\labelenumi}{(\roman{enumi})} 
\item $\rho$ from Proposition \ref{ramsey-prop}, the Ramsey lemma;
\item $\kappa$ and $p_2$ from Proposition \ref{prop99}, the combined counting and Ramsey lemmas;
\item $\beta$ from Lemma \ref{bohr-lower} (our lower bound for the density of a $\QM$-Bohr set);
\item $M_{\Omega}$, $D_{\Omega}$ and $p_\Omega$ from Proposition \ref{regularity}, the regularity lemma, depending on some soon-to-be-defined growth function $\Omega$;
\item $\nu$ from Corollary \ref{gvn-cor}, the generalised von Neumann theorem.
\end{enumerate}
We use these functions to define a growth function $\Omega$ by
\[ \Omega_{r,d,\delta}(R,d'):= 1/\nu\bigg( \textstyle\frac{1}{372}\displaystyle \eta\big(\min\{\kappa(r,d,\delta,R),\delta\},d',r-1\big) \rho(r,d,\delta) \beta\big(d',\min\{\kappa(r,d,\delta,R),\delta\}\big)  \bigg)\] 
The monotonicity of $\kappa$, $\rho$, $\beta$ and $\eta(\cdot,\cdot,r-1)$ ensures that $\Omega$ is monotonic (as a function from $\Z_{\geq 0} \times \Z_{\geq 0} \times (0,1] \times \R_{\geq 0} \times \Z_{\geq 0} \rightarrow \R_{\geq 1}$), and so we may apply the regularity lemma (Proposition \ref{regularity}).  Doing so gives up monotonic functions $M_\Omega$ and $D_\Omega$, and with these in hand we make some definitions:
\[ \eps := \textstyle\frac{1}{1728}\displaystyle \rho(r,d,\delta),  M_{r,d,\delta}:=M_\Omega(r,d,\delta,\eps) \text{ and } D_{r,d,\delta}:=D_\Omega(r,d,\delta,\eps),\]
all of which values depend monotonically on the values of $r$, $d$, and $\delta$ by the monotonicity of $\rho$, $M_\Omega$ and $D_\Omega$.  Finally we write
\[ \delta':=\min\{\kappa(r, d,\delta,M_{r,d,\delta}),\delta\},\]
which depends monotonically on $r$, $d$, and $\delta$ by the monotonicity of $\kappa$ and $M_{r,d,\delta}$.  We then define the functions $\eta( \cdot, \cdot, r),\zeta(\cdot,\cdot,r)$ and $p_0(\cdot,\cdot,r)$ as follows:
\[ \eta(\delta,d,r) := \min \bigg\{ \textstyle\frac{1}{100}\displaystyle \eps^2, \textstyle\frac{1}{2}\displaystyle\beta(D_{r,d,\delta},\delta') \eta\big(\delta',D_{r,d,\delta} , r-1\big) ,\eta(\delta,d,r-1)\bigg\},\]
\[ \zeta(\delta,d,r) := \min \bigg\{ \zeta(\delta,d,r-1), \zeta(\delta',D_{r,d,\delta},r-1), \textstyle\frac{1}{128}\displaystyle\eta(\delta',D_{r,d,\delta},r-1) \rho(r,d,\delta) \beta(D_{r,d,\delta},\delta')\bigg\},\]
and
\begin{align*}
p_0(\delta,d,r) &:= \left\lceil\max \bigg\{ p_2\left(M_{r,d,\delta},D_{r,d,\delta},r,\textstyle\frac{1}{2}\displaystyle\eta(\delta',D_{r,d,\delta},r-1)\beta(D_{r,d,\delta},\delta'),\delta\right),\right.\\
& \qquad \qquad \qquad 1+1/\nu\left(1/\Omega_{r,d,\delta}(M_{r,d,\delta},D_{r,d,\delta})\right),p_0(\delta',D_{r,d,\delta},r-1),p_\Omega(r,d,\delta,\eps),\\
& \left.\qquad \qquad \qquad 384(1+M_{r,d,\delta})^3\eta(\delta',D_{r,d,\delta},r-1)^{-1}\beta(D_{r,d,\delta},\delta')^{-1}\rho(r,d,\delta)^{-1}\bigg\}\right\rceil.
\end{align*}
These functions inherit the appropriate monotonicity properties from the monotonicity of $\eps$, $D_{r,d,\delta}$, $\delta'$, $\beta$, $\eta(\cdot,\cdot,r-1)$, $\zeta(\cdot,\cdot,r-1)$, $\rho$, $p_2$, $p_0(\cdot,\cdot,r-1)$, $M_{r,d,\delta}$, and $\nu$.

Now suppose that $B$ is a $d$-dimensional $\QM$-Bohr set of width $\delta$ with underlying $\QM$-system $\Psi$, and $c : \tilde B \rightarrow [r]$ is a colouring of $\tilde{B}$ where $|\tilde B| \geq (1 -\eta(\delta,d,r) )|B|$.  Since $\eta(\delta,d,r) \leq \frac{1}{100}\eps^2$, the regularity lemma tells us that (since $p \geq p_\Omega(r,d,\delta,\eps)$) there is a $\QM$-system $\Psi' \supset \Psi$, functions $F_1,\dots, F_r : \G^{d'} \rightarrow \R_{\geq 0}$ and functions $g_1,\dots, g_r : \F \rightarrow [-1,1]$ such that:

\begin{enumerate}
\item \label{1} $\Vert F_i \Vert_{\trig} \leq M_\Omega(r,d,\delta,\eps)=M_{r,d,\delta}$, $i \in \{ 1,\dots, r\}$;
\item \label{2} $d' \leq D_\Omega(r, d,\delta,\eps)=D_{r,d,\delta}$;
\item \label{3} $\Vert F_i \circ \Psi' - g_i \Vert_2 \leq \eps$ for $i \in \{1,\dots, r\}$;
\item \label{4} $\Vert 1_{c^{-1}(i)} - g_i\Vert_{\QM} \leq 1/\Omega_{r,d,\delta}(\Vert F_i \Vert_{\trig}, d')$;
\item \label{5} $\sum_{i = 1}^r F_i \circ \Psi' \geq 1$ pointwise on $B(\Psi, \frac{1}{2}\delta)$;
\item \label{6} $\|F_i \circ \Psi'\|_4 \leq 2$ for all $i \in \{1,\dots,r\}$.
\end{enumerate}

For each $i \in [r]$ write $\delta_i := \min\{\kappa(r, d,\delta,\|F_i\|_{\trig}),\delta\}$ (so that $\delta_i \geq \delta'$ for all $i$), $B_i:=B(\Psi',\delta_i)$ and $S_i:=c^{-1}(i)\cap B_i$.  (On a first pass it may seem like we could take $\delta_i=\delta'$ for all $i$, but we cannot because in (\ref{4}) we only have an upper bound in terms of $\|F_i\|_{\trig}$ rather than $M_{r,d,\delta}$, and the former might be much smaller than the latter.)  We consider two cases.

Suppose first that $|S_i| \leq \frac{1}{2}\eta(\delta_i,d',r-1) |B_i|$ for some $i$.  Then by the lower bound for the density of $\QM$-Bohr sets (Lemma \ref{bohr-lower}); the definition of $\eta$; the fact that $d' \leq D_{r,d,\delta}$ and $\delta_i \geq \delta'$; and the monotonicity of $\eta(\cdot,\cdot,r-1)$ and $\beta$, we have
\begin{align*}
|B_i \setminus \tilde B| \leq |B \setminus \tilde B| \leq \eta(\delta, d, r) |B| & \leq \frac{\eta(\delta,d, r)}{\beta(d',\delta_i)} |B_i|\\ & \leq  \frac{\eta(\delta',D_{r,d,\delta},r-1)\beta(D_{r,d,\delta},\delta')}{2\beta(d',\delta_i)} |B_i|\\ &  \leq \textstyle\frac{1}{2}\displaystyle\eta(\delta_i, d', r-1) |B_i|.
\end{align*}
It follows that $c$ restricts to an $(r-1)$-colouring of $B_i\setminus (\tilde{B}\cup S_i)$ where, by the triangle inequality, we have
\begin{equation*}
|B_i\setminus (\tilde{B}\cup S_i)| \geq |B_i| - |B_i\setminus \tilde{B}| - |S_i| \leq \eta(\delta_i,d',r-1) |B_i|.
\end{equation*}
By the inductive hypothesis and monotonicity of $\zeta(\cdot,\cdot,r-1)$ we conclude that there are at least
\begin{equation*}
\zeta(\delta_i,d',r-1)p^2 \geq \zeta(\delta',D_{r,d,\delta},r-1)p^2
\end{equation*}
pairs $(x,y)$ with $x,y,x+y,xy$ all the same colour.  Thus in this case the result is proved.

The second case is that $|S_i| \geq \frac{1}{2}\eta(\delta_i,d',r-1) |B_i|$ for all $i$.  In this case we have
\begin{align}
\nonumber \mu_\F(S_i) & \geq \textstyle\frac{1}{2}\displaystyle\eta(\delta_i,d',r-1)\mu_\F(B_i) \\ \nonumber & \geq \textstyle\frac{1}{2}\displaystyle\eta(\delta_i,d',r-1)\beta(d',\delta_i) \\ \label{e1} & \geq \textstyle\frac{1}{2}\displaystyle\eta(\delta',D_{r,d,\delta},r-1)\beta(D_{r,d,\delta},\delta')
\end{align}
by the lower bound for the density of $\QM$-Bohr sets (Lemma \ref{bohr-lower}), the monotonicity of $\eta(\cdot,\cdot,r-1)$ and $\beta$, and the fact that $d' \leq D_{r,d,\delta}$ and $\delta_i \geq \delta'$.

By Proposition \ref{prop99} (for which application we need (\ref{5})) there is some $i \in \{1,\dots,r\}$ such that
\[ T(F_i \circ \Psi', 1_{S_i}, F_i \circ \Psi', F_i \circ \Psi') \geq \textstyle\frac{1}{16}\displaystyle\mu_{\F}(S_i) \rho(r,d,\delta),\]
provided
\begin{equation*}
p\geq p_2\left(M_{r,d,\delta},d',r, \textstyle\frac{1}{2}\displaystyle\eta(\delta',D_{r,d,\delta},r-1)\beta(D_{r,d,\delta},\delta'),\delta\right),
\end{equation*}
which follows from the definition of $p_0$ and the monotonicity of $p_2$.

Note that by (\ref{6}) and the triangle inequality we have
\begin{equation*}
\|g_i\|_4 \leq 1, \|F_i\circ \Psi'\|_4 \leq 2, \text{ and }\|g_i - F_i \circ \Psi'\|_4 \leq 3.
\end{equation*}
Thus by the triangle inequality, Lemma \ref{simple-lem} and item (\ref{3}) above we have
\begin{align*}
& |T(g_i, 1_{S_i}, g_i, g_i) - T(F_i \circ \Psi', 1_{S_i}, F_i \circ \Psi', F_i \circ \Psi')|\\
& \leq  |T(g_i, 1_{S_i}, g_i, g_i) - T(F_i \circ \Psi', 1_{S_i}, g_i,g_i )|\\
& \qquad +  |T(F_i \circ \Psi', 1_{S_i}, g_i, g_i) - T(F_i \circ \Psi', 1_{S_i}, F_i \circ \Psi',g_i)|\\
& \qquad +  |T(F_i \circ \Psi', 1_{S_i}, F_i \circ \Psi', 1_{S_i},  g_i) - T(F_i \circ \Psi', 1_{S_i}, F_i \circ \Psi',F_i \circ \Psi')|\\
& \leq \frac{3(1+\|F_i\|_\infty)^3}{p} + 27\mu_{\F}(S_i)\eps \leq  \frac{3(1+M_{r,d,\delta})^3}{p} + 27\mu_{\F}(S_i)\eps.
\end{align*}
By the choice of $\eps$, and provided
\begin{equation*}
p \geq 3(1+M_{r,d,\delta})^3 \cdot 64 \mu_\F(S_i)\rho(r,d,\delta)^{-1}
\end{equation*}
this tells us that
\begin{equation}\label{eqn.s}
|T(g_i, 1_{S_i}, g_i, g_i) - T(F_i \circ \Psi', 1_{S_i}, F_i \circ \Psi', F_i \circ \Psi')| \leq \textstyle\frac{1}{32}\displaystyle\mu_{\F}(S_i)\rho(r,d,\delta).
\end{equation}
Of course by \eqref{e1} we see
\begin{align*}
& 3(1+M_{r,d,\delta})^3 \cdot 64 \mu_\F(S_i)\rho(r,d,\delta)^{-1} \\
& \qquad \qquad \leq  384(1+M_{r,d,\delta})^3\eta(\delta',D_{r,d,\delta},r-1)^{-1}\beta(D_{r,d,\delta},\delta')^{-1}\rho(r,d,\delta)^{-1} \leq p_0(r,d,\delta),
\end{align*}
and so \eqref{eqn.s} holds provided $p \geq p_0(r,d,\delta)$.

It follows from the hypothesis on the size of the $S_i$s, and the lower bound on the density of $\QM$-Bohr sets that
\begin{align*}
T(g_i, 1_{S_i}, g_i, g_i) & \geq \textstyle\frac{1}{32}\displaystyle \mu_{\F}(S_i) \rho(r,d,\delta)\\
 & \geq \textstyle\frac{1}{64}\displaystyle\mu_{\F}(B_i) \eta(\delta_i,d',r-1) \rho(r,d,\delta)\\ & \geq \textstyle\frac{1}{64}\displaystyle \eta(\delta_i,d',r-1) \rho(r,d,\delta) \beta(d',\delta_i).
\end{align*}

By the generalised von Neumann theorem (Corollary \ref{gvn-cor}), monotonicity of $\nu^{-1}$, the choice of $\Omega$ and (\ref{4}) above, this implies that
\begin{align*}
& |T(g_i, 1_{S_i}, g_i, g_i) -T(1_{c^{-1}(i)}, 1_{S_i}, 1_{c^{-1}(i)}, 1_{c^{-1}(i)})|\\
&  \qquad \leq |T(g_i-1_{c^{-1}(i)}, 1_{S_i}, g_i, g_i)|\\
& \qquad \qquad \qquad  + |T(1_{c^{-1}(i)}, 1_{S_i}, g_i-1_{c^{-1}(i)}, g_i)| +  |T(1_{c^{-1}(i)}, 1_{S_i}, 1_{c^{-1}(i)},g_i-1_{c^{-1}(i)})|\\
&  \qquad \leq 3\nu^{-1}(\|g_i-1_{c^{-1}(i)}\|_{\QM})\\
&  \qquad \leq 3\nu^{-1}(1/\Omega_{r,d,\delta}(\|F_i\|_{\trig},d'))  = \textstyle\frac{1}{128}\displaystyle\eta(\delta_i,d',r-1) \rho(r,d,\delta) \beta(d',\delta_i),
\end{align*}
provided
\begin{equation*}
p > 1/\nu\left(1/\Omega_{r,d,\delta}(\|F_i\|_{\trig},d')\right).
\end{equation*}
By monotonicity of $\Omega_{r,d,\delta}$ and the fact that $d' \leq D_{r,d,\delta}$ and $\|F_i\|_{\trig} \leq M_{r,d,\delta}$, this follows from $p \geq p_0(\delta,d,r)$ as defined.  By the monotonicity of $\eta(\cdot,\cdot,r-1)$ and $\beta$, and the pointwise inequality $1_{S_i} \leq 1_{c^{-1}(i)}$ we conclude that
\begin{align*}
T(1_{c^{-1}(i)}, 1_{S_i}, 1_{c^{-1}(i)}, 1_{c^{-1}(i)}) & \geq \textstyle\frac{1}{128}\displaystyle \eta(\delta_i,d',r-1) \rho(r,d,\delta) \beta(d',\delta_i)\\
& \geq \textstyle\frac{1}{128}\displaystyle\eta(\delta',D_{r,d,\delta},r-1) \rho(r,d,\delta) \beta(D_{r,d,\delta},\delta').
\end{align*}
The result is proved given the choice of $\zeta$.
\end{proof}

\emph{Remark.} The above argument is not straightforward, in particular with regard to checking that the parameters do not depend on one another in a circular manner. A different way to arrange the arguments might be via the use of an ultraproduct. However, this introduces a considerable amount of additional language, which propagates out to other sections as well. Therefore, even though on some conceptual level an ultraproduct formulation could be the ``right'' way to phrase the argument, we have chosen not to follow this route.

\section{Proof of the regularity lemma}\label{sec5}

In this section we prove the regularity lemma, Proposition \ref{regularity}. The reader may wish to recall its statement. The proof proceeds using an ``energy-increment'' argument of a type that will be familiar to experts. However, it takes some effort to sort out the technical details specific to our situation. Let $d$ be a non-negative integer (which $d$ we are talking about at any given point will be clear from context). We begin by describing a partition of $\G^d$ into certain boxes. Let $R > 0$ be a power of $2$. Suppose that $t,u,v \in \{0,1,\dots, R-1\}^d$. Then we define \emph{generalised intervals}
\begin{align*}
I_{R;t,u,v}  & := \big\{ (\theta, \phi, z) \in \G^d : \theta_j \in \big[\frac{t_j}{R} + \sqrt{2},  \frac{t_j+1}{R} + \sqrt{2}\big) +\Z,  \phi_j \in \big[\frac{u_j}{R} + \sqrt{2}, \frac{u_j+1}{R} + \sqrt{2}\big)+\Z ,\\ 
 & \frac{1}{2\pi i}\log z_j \in \big[\frac{v_j}{R} + \sqrt{2}, \frac{v_j+1}{R} + \sqrt{2}\big)+\Z \quad \mbox{for}\quad  j \in \{1,\dots,d\}\big\} .
\end{align*}
It is worth making a few remarks about these sets.
\begin{enumerate}
\renewcommand{\labelenumi}{(\roman{enumi})} 
\item For each $R$ the set $\mathcal{I}_R:=\{I_{R;t,u,v}: t,u,v \in \{0,1,\dots, R-1\}^d\}$ is a partition of $\G^d$ into $R^{3d}$ sets.
\item We restrict $R$ to be a power of $2$ as a technical convenience, to ensure that if $R'>R$ then $\mathcal{I}_{R'}$ is a refinement of $\mathcal{I}_{R}$.
\item The $\sqrt{2}$ here is present as a technical device to help control edge effects later on; it has the usual property of being poorly approximated by rationals. The only point in the argument at which it is relevant is in the proof of estimate \eqref{edge-bound-3}, a technical point in the proof of Lemma \ref{lem5.4}. 
\end{enumerate}
Let $\Psi = (a_i x^2/p, 2a_i x/p, \psi_i(x))_{i = 1}^d$ be a $\QM$-system of dimension $d$. Suppose that $R > 0$ is a power of $2$.  Where $\Psi$ is clear like this we write
\begin{equation*}
A_{R;t,u,v}:=\Psi^{-1}(I_{R;t,u,v}) \text{ for each } t,u,v \in \{0,\dots,R-1\}^d.
\end{equation*}
The sets $\{A_{R;t,u,v}: t,u,v \in  \{0,\dots,R-1\}^d\}$ form a partition of $\F$ (possibly with the addition of the empty set) since $\mathcal{I}_R$ is a partition of $\G^d$ and, in particular, generate a $\sigma$-algebra $\mathcal{B}$.  We define the associated \emph{projection operator} at \emph{resolution} $R$ to be
\begin{equation*}
\Pi_R^{\Psi} : L^{2}(\F) \rightarrow L^{2}(\F); f \mapsto \E(f | \mathcal{B}),
\end{equation*}
so that
\[ \Pi_R^{\Psi} f(x) = \frac{1}{|A(x)|} \sum_{x' \in A(x)} f(x')\] where $A(x)$ is the atom containing $x$. As with any conditional expectation operator, $\Pi_R^{\Psi}$ is self-adjoint. Indeed
\[
\langle f, \E(g |\mathcal{B}) \rangle  = \E_x f(x) \frac{1}{|A(x)|} \sum_{x' \in A(x)} g(x')  = \frac{1}{p} \sum_{x, x'} f(x) g(x') 1_{A(x)}(x') \frac{1}{|A(x)|},
\]
and the kernel $1_{A(x)}(x') \frac{1}{|A(x)|}$ is symmetric in $x$ and $x'$ so this also equals $\langle \E(f | \mathcal{B}), g\rangle$.

\begin{lemma}\label{lem5.2}
Suppose that $f$ has $\|f\|_{\infty} \leq 1$ and $\Vert f \Vert_{\QM} \geq \delta$. Let $R \geq C\delta^{-1}$ be a power of $2$. Then there is a $\QM$-system $\Phi$ of dimension at most 2 and a function $g \in L^{\infty}(\F)$ with $\Vert g \Vert_{\infty} \leq 1$, such that $|\langle f, \Pi_R^{\Phi} g \rangle| \gg \delta$.
\end{lemma}
\emph{Remark.} The functions of the form $\Pi^{\Phi}_R g$ where $\Vert g \Vert_{\infty} \leq 1$, are precisely those functions which are bounded by $1$ in modulus and are constant on atoms $A_{R;t,u,v}$.
\begin{proof}
It follows from the definition of the $\QM$-norm (Definition \ref{def.qmnorm}) that, under the hypotheses on $f$, there are $a_1, a_2$ and a multiplicative function $\psi$ such that 
\[ |\E_x f(x) \overline{e_p(a_1x^2 + 2a_2x) \psi(x)}| \geq \delta.\]
Consider the $\QM$-system $\Phi = (a_i x^2/p, 2a_i x/p, \psi_i(x))_{i = 1,2}$ in which $\psi_1 = \psi_2 = \psi$. Then
\[ \E_x f(x) \overline{e_p(a_1x^2 + 2a_2x) \psi(x)} = \langle f, g\rangle,\]
where $g = F \circ \Phi$ with $F : \G^2 \rightarrow \C$ defined by
\[ F(\theta_1, \theta'_1, z_1; \theta_2, \theta'_2, z_2) := e(\theta_1 + \theta'_2)z_1.\]
Thus
\begin{equation}\label{eq521} |\langle f, F \circ \Phi \rangle| \geq \delta.\end{equation}

Now since $F$ is quite a smooth function, we have $g \approx \Pi_R^{\Psi} g$. More precisely,
\[ \Vert \Pi_R^{\Phi} g - g\Vert_{\infty}  \leq \sup_{x' \in A(x)} |F(\Phi(x')) - F(\Phi(x))|  \ll R^{-1},\]
since the Lipschitz constant of $F$ is $O(1)$.
It follows from this and \eqref{eq521} that if $R > C\delta^{-1}$ with $C$ large enough then $|\langle f, \Pi_R^{\Phi} g\rangle| \gg \delta$, and the result follows.
\end{proof}

The next lemma is a result of ``Koopman von Neumann'' type. In establishing this we shall need the following property of the projection operators $\Pi^{\Psi}_R$: if $\Psi' \supset \Psi$ and $R | R'$, then
\begin{equation}\label{nesting} \Pi^{\Psi}_R \Pi^{\Psi'}_{R'} f = \Pi^{\Psi}_{R} f.\end{equation}
This is because each atom associated to $\Pi_R^{\Psi}$ is a union of atoms associated to $\Pi_{R'}^{\Psi'}$. 

\begin{lemma}\label{kvn}
Suppose that $f_1,\dots, f_r : \F \rightarrow \C$ are such that $\|f_i\|_\infty \leq 1$ for all $i \in \{1,\dots,r\}$. Let $\Psi$ be a $\QM$-system, and let $R \geq C\delta^{-1}$ be a power of $2$.  Then there is a $\QM$-system $\Psi' \supset \Psi$ with $\dim\Psi'\leq \dim\Psi + O( r\delta^{-2})$ such that $\Vert f_i - \Pi_{R}^{\Psi'} f_i\Vert_{\QM} \leq \delta$ for all $i \in\{ 1,\dots, r\}$.
\end{lemma}
\begin{proof}
 Define a nested sequence $\Psi =: \Psi_0 \subset \Psi_1 \subset \Psi_2 \subset \dots$ of $\QM$-systems with $\dim \Psi_j = \dim \Psi + 2j$ in the following manner. For $j = 0,1,2,3,\dots$, proceed as follows. Set $f_{i,j} := f_i - \Pi_R^{\Psi_j} f_i$. If $\Vert f_{i,j} \Vert_{\QM} \leq \delta$ for $i \in \{ 1,\dots, r\}$ then stop; otherwise, by Lemma \ref{lem5.2}, there is an $i \in \{1,\dots,r\}$, some $\QM$-system $\Phi$ of dimension $2$ and a function $g \in L^{\infty}(\F)$ with $\Vert g \Vert_{\infty} \leq 1$, such that $|\langle f_{i,j}, \Pi_R^{\Phi} g \rangle| \gg \delta$. In this case, set $\Psi_{j+1} := \Psi_j \cup \Phi$ and note by idempotence of $\Pi_R^\Phi$ and (\ref{nesting}) we have
 \begin{equation*}
 \langle f_i - \Pi_R^{\Psi_{j+1}}f_i,\Pi_R^{\Phi} g\rangle = \langle \Pi_R^{\Phi} f_i - \Pi_R^{\Phi}\Pi_R^{\Psi_{j+1}}f_i, \Pi_R^{\Phi} g\rangle=  0,
 \end{equation*}
 and hence
 \begin{equation*}
 \langle \Pi_R^{\Psi_{j+1}}f_i - \Pi_R^{\Psi_{j}}f_i,\Pi_R^{\Phi} g\rangle = \langle f_{i,j},\Pi_R^{\Phi} g\rangle -  \langle f_i - \Pi_R^{\Psi_{j+1}}f_i,\Pi_R^{\Phi} g\rangle  \gg \delta.
 \end{equation*}
By the Cauchy-Schwarz inequality we conclude
\[ \Vert \Pi_R^{\Psi_{j+1}} f_i - \Pi_R^{\Psi_{j+1}} f_i \Vert_2 \gg \delta,\]
and so (expanding out the $L_2$-norm square and using (\ref{nesting})) it follows that 
\[ \Vert \Pi_R^{\Psi_{j+1}}f_i \Vert_2^2 - \Vert \Pi_R^{\Psi_j} f_i\Vert_2^2 = \Vert \Pi_R^{\Psi_{j+1}}f_i -  \Pi_R^{\Psi_j}f_i \Vert_2^2 \gg \delta^{2}.\]
Hence, defining the energy
\[ E_j := \sum_{i = 1}^r \Vert \Pi_R^{\Psi_j} f_i \Vert_2^2,\] we have
\[ E_{j+1} - E_j \gg \delta^{2}.\]
As long as this process continues, we obviously have the trivial bound $E_j \leq r$. It follows that the process terminates after at most $O(r\delta^{-2})$ steps, and the result follows. 
\end{proof}

For $f:\F \rightarrow [0,1]$, the function $\Pi_R^{\Psi} f$ is, by definition, constant on the atoms $A_{R;t,u,v}$. In particular it factors through $\G^d$ so we have $\Pi_R^{\Psi} f(x) = F_0 \circ \Psi(x)$, for a certain function $F_0 : \G^d \rightarrow [0,1]$. We shall need control of the trig-norm of $F_0$ in applications which, as things stand, may be infinite.  Since the characters of $\G^d$ form a (Schauder) basis for the functions on $\G^d$, the characters composed with $\Psi$ span all functions $\F \rightarrow [0,1]$.  This space of functions is finite dimensional and so the $F_0$s above may be taken to have finite trig-norm.  Making this bound on the trig-norm uniform in the function requires an additional argument, and we need a bound uniform in the function \emph{and} $p$.  This is a routine, if technical, endeavour. 
\begin{lemma}\label{lem5.4}
There are monotonic functions $M_0:(0,1]\times \Z_{\geq 0} \times \R_{>0}\rightarrow \R_{\geq 1}$, and $p_3:(0,1] \times \R_{>0}\times \Z_{\geq 0} \rightarrow \N$ such that if $p \geq p_3(\eps,d,R)$ then the following holds.

For all $f : \F \rightarrow [0,1]$, $d$-dimensional $\QM$-systems $\Psi$, and parameters $R > 0$ a power of $2$, and $\eps \in (0,1]$ there is a function $F : \G^d \rightarrow \R_{\geq 0}$ such that  
\begin{enumerate}
\item[(i)] $F \circ \Psi \geq \Pi_R^{\Psi}f$ pointwise;
\item[(ii)] $\Vert F \Vert_{\trig} \leq M_0(\eps, d, R)$;
\item[(iii)] $\Vert F \circ \Psi - \Pi_R^{\Psi}f \Vert_2 \leq \eps/2$, and
\item[(iv)] $\Vert F\circ \Psi\Vert_4 \leq 2$.
\end{enumerate}
\end{lemma}
\begin{proof}
Set 
\begin{equation}\label{eta-def}
\eta := \left(\frac{\eps}{100 (1+2R^{3d})^4R d}\right)^4.
\end{equation}
For each $t,u,v \in \{0,1,\dots, R-1\}^d$ we define an ``$\eta$-enlargement'' and an ``$\eta$-reduction'' of $I_{R;t,u,v}$ by
\begin{align*} &
I_{R;t,u,v}^\pm  := \big\{ (\theta, \phi, z) \in \G^d : \big\|\theta_j -\frac{2t_j+1}{2R} - \sqrt{2}\big\|_{\R/\Z}  < \frac{1}{2R} \pm \eta, \big\|\phi_j -\frac{2u_j+1}{2R} - \sqrt{2}\big\|_{\R/\Z}  < \frac{1}{2R} \pm \eta, \\ & \qquad \big\| \frac{1}{2\pi i}\log z_j -\frac{2v_j+1}{2R} - \sqrt{2}\big\|_{\R/\Z}  < \frac{1}{2R} \pm \eta \big\} \quad \mbox{for} \quad j \in \{1,\dots,d\} \big\}.
\end{align*}

We then define the ``boundary'' to be
\begin{equation*}
E:=\bigcup\{I_{R;t,u,v}^+ \setminus I_{R;t,u,v}^{-}: t,u,v \in \{0,\dots,R-1\}^d\},
\end{equation*}
and make two claims:\vspace*{8pt}

\noindent \textbf{Claim A:} If $p \geq C\eta^{-4}$ then 
\begin{equation*}
|\Psi^{-1}(E)| \leq \eps^2 p/10(1+2R^{3d})^4;
\end{equation*}\vspace*{8pt}

\noindent \textbf{Claim B:} 
If $z \in I_{R;t,u,v}^-\cap I_{R;t',u',v'}^+$ then $(t',u',v')=(t,u,v)$. \vspace*{8pt}

We shall establish these claims later, but give the rest of the proof first.

Since the functions on $\G^d$ with bounded trig norm are dense in $C(\G^d)$ we see that there are functions $F_{R;t,u,v} : \G^d \rightarrow \R$ and a function $M_*$ with the following properties:
\begin{enumerate}
\item $0 \leq F_{R;t,u,v} \leq 1+\eps/10R^{3d}$ pointwise;
\item $F_{R;t,u,v}(z) \geq 1$ for all $z \in I_{R;t,u,v}$;
\item $|F_{R;t,u,v}(z)| \leq \eps/10R^{3d}$ for all $z \not \in I_{R;t,u,v}^+$;
\item $\Vert F_{R;t,u,v} \Vert_{\trig} \leq M_*(\eps, d, R)$.
\end{enumerate}
The function $M_*$ need not be monotonic but this can be quickly fixed.  For reasons that will become clear later we shall in fact put
\[
M_0(\eps,d,R) := \max\{1,R^{3d}\max\{M_*(2^{-l}, d_*, R_*): l,d_*,R_* \in \N, 2^{1-l} \geq \eps, d_* \leq d, R_* \leq R\}\},
\]
which \emph{is} monotonic.

The function $\Pi_R^{\Psi} f$ is constant on atoms $A_{R;t,u,v}$. If $A_{R;t,u,v}$ is non-empty then write $\lambda_{R;t,u,v}$ for the value of $\Pi_R^{\Psi} f$ on this set and if $A_{R;t,u,v}=\emptyset$ we set $\lambda_{R;t,u,v}=0$, so that the $\lambda_{R;t,u,v}$s are all non-negative. Then we define
\[ F := \sum_{t,u,v \in \{0,1,\dots,R-1\}^d } \lambda_{R;t,u,v} F_{R;t,u,v},\]
and
\[ F_0 := \sum_{t,u,v \in \{0,1,\dots, R-1\}^d} \lambda_{R;t,u,v} 1_{I_{R;t,u,v}},\] 
It follows that
\[ \Pi_R^{\Psi} f = F_0 \circ \Psi,\]
and since the $F_{R;t,u,v}$s and the $\lambda_{R;t,u,v}$s are non-negative, by (2) above we have
\begin{equation*}
F = \sum_{t,u,v \in \{0,1,\dots,R-1\}^d } \lambda_{R;t,u,v} F_{R;t,u,v} \geq  \sum_{t,u,v \in \{0,1,\dots, R-1\}^d} \lambda_{R;t,u,v} 1_{I_{R;t,u,v}} = F_0.
\end{equation*} 
We now show that $F$ satisfies the relevant properties. 
\begin{enumerate}
\item[(i)] This follows immediately since $F \geq F_0$ pointwise and so $F\circ \Psi(x) \geq F_0\circ \Psi(x) =\Pi_R^{\Psi}f(x)$.
\item[(ii)] Since $\Vert \cdot \Vert_{\trig}$ satisfies (\ref{subadd}) and (\ref{scale}), and $|\lambda_{R;t,u,v}| \leq 1$, (4) tells us that
\begin{equation*}
\|F\|_{\trig} \leq \sum_{t,u,v}{\max\{1,|\lambda_{R;t,u,v}|\}\|F_{R;t,u,v}\|_{\trig}} \leq R^{3d}M_*(\eps,d,R) \leq M_0(\eps,d,R)
\end{equation*}
as required.
\item[(iii)]  Suppose that $z \not \in E$.  Then, since $\{ I_{R;t,u,v}^+: t,u,v\in \{0,\dots, R-1\}^d\}$ covers $\G^d$ we see that $z \in I_{R;t,u,v}^-$ for some $t,u,v \in \{0,\dots,R-1\}^d$, and so
\begin{equation*}
F_0(z)=\lambda_{R;t,u,v} 1_{I_{R;t,u,v}}(z).
\end{equation*}
By Claim B we have that $z \not \in I_{R;t',u',v'}^+$ for any $(t',u',v')\neq (t,u,v)$.  Hence by (3) we have $|F_{R;t',u',v'}(z)| \leq \eps/10R^{3d}$ whenever $(t',u',v') \neq (t,u,v)$, and so
\begin{align*}
|F(z) - F_0(z)| & \leq |\lambda_{R;t,u,v}(F_{R;t,u,v}(z)- 1_{I_{R;t,u,v}}(z))|\\
& \qquad + \big|\sum_{(t',u',v') \neq (t,u,v) } \lambda_{R;t',u',v'} F_{R;t',u',v'}(z)\big| \leq \textstyle\frac{1}{10}\displaystyle \eps.
\end{align*}
On the other hand, if $z \in E$ then we just have the trivial bound
\begin{equation*}
|F(z)-F_0(z)| \leq \big|\sum_{(t,u,v) } \lambda_{R;t,u,v} F_{R;t,u,v}\big| +|F_0(z)| \leq 2R^{3d}+1.
\end{equation*}
It follows that
\begin{equation*}
\|F\circ \Psi - f\|_{2}^2 = \|F\circ \Psi - F_0 \circ \Psi\|_{2}^2 \leq \E_{x \in \F}{\left((1+2R^{3d})^21_E(\Psi(x)) + (\eps/10)^2\right)},
\end{equation*}
and we have (iii) by Claim A.
\item[(iv)] Finally, using the above we have
\begin{align*}
\|F\circ \Psi \|_{4} & \leq \|F_0\circ \Psi\|_4 + \|F\circ \Psi-F_0\circ \Psi\|_4\\
& \leq 1 + \left( \E_{x \in \F}{(1+2R^{3d})^41_E(\Psi(x))} + (\eps/10)^4\right)^{1/4} \leq 2,
\end{align*}
from which (iv) follows.
\end{enumerate}
A suitable choice of $p_3$ can be made given the definition of $\eta$ and the hypothesis of Claim A.  We now turn to establishing the claims.
\begin{proof}[Proof of Claim A]
If $x \in \Psi^{-1}(E)$ then there are some $t,u,v \in \{0,\dots,R-1\}^d$ such that
\begin{equation*}
\Psi(x) \in I_{R;t,u,v}^+ \setminus I_{R;t,u,v}^-,
\end{equation*}
and so it is enough to establish the following statements for any $a \in \F^*$, non-trivial $\psi \in \widehat{\F^*}$ and for any interval $J \subset \R/\Z$ of the form $J = \frac{j}{R} + \sqrt{2} + (-\eta, \eta)+\Z$, $j \in \{0,1\dots, R-1\}$, we have
\begin{equation}\label{edge-bound-1}
\# \{ x \in \F : ax^2/p \in J \} \leq \frac{\eps^2 p}{100Rd(1+2R^{3d})^4},
\end{equation}
\begin{equation}\label{edge-bound-2}
\# \{ x \in \F : 2ax/p \in J \} \leq \frac{\eps^2 p}{100Rd(1+2R^{3d})^4},
\end{equation}
and
\begin{equation}\label{edge-bound-3}
\# \{ x \in \F : \frac{1}{2\pi i}\log \psi(x) \in J \} \leq \frac{\eps^2 p}{100Rd(1+2R^{3d})^4}.
\end{equation}
The claim then follows from allowing $j,a,\psi$ to range over all choices from the sets  $\{0,\dots, R-1\}$, $\{a_1,\dots, a_d\}$ and $\{\psi_1,\dots, \psi_d\}$ respectively. Of course we must still establish \eqref{edge-bound-1}, \eqref{edge-bound-2} and \eqref{edge-bound-3}.

\emph{Proof of \eqref{edge-bound-2}.} This is straightforward, as $2ax/p$ takes all the values $r/p$, $r \in \{ 0,1,\dots, p-1\}$, precisely once and so we have 
\[ \# \{x \in \F : 2ax/p \in J\} = p|J| + O(1).\] \eqref{edge-bound-2} follows immediately provided $p \geq C\eta^{-1}$.

\emph{Proof of \eqref{edge-bound-1}.} This follows in more-or-less the same way as \eqref{edge-bound-2}, using instead the fact that $ax^2/p$ takes each value $r/p$, $r \in \{0,1,\dots, p-1\}$, at most twice.

\emph{Proof of \eqref{edge-bound-3}.} 
The image of $\F$ under $\frac{1}{2\pi i}\log \psi$ is $\{0, \frac{1}{Q},\dots, \frac{Q-1}{Q}\}$ for some $Q$, and each point is hit the same number $\frac{p-1}{Q}$ of times as $x$ ranges over $\F^*$. The number of the points $\{0,\frac{1}{Q},\dots, \frac{Q-1}{Q}\}$ lying in $J$ is at most $1 + 2\eta Q$, and so \begin{equation}\label{first} \# \{x \in \F: \frac{1}{2\pi i} \log \psi(x) \in J\} \leq \frac{p-1}{Q}(1 + 2\eta Q) + 1.\end{equation}
Without a lower bound on $Q$, this is useless. To obtain a lower bound on $Q$, we assume that there exists at least one $x \in \F^*$ for which $\frac{1}{2\pi i}\log \psi(x) \in J$.  (Otherwise \eqref{edge-bound-3} really is trivial, provided $p \geq C\eta^{-1}$.) Suppose that for this point we have $\frac{1}{2\pi i}\log \psi(x) = \frac{q}{Q}$. Then we have $\Vert\frac{q}{Q} - \frac{j}{R} - \sqrt{2}\Vert \leq \eta$, and so $\Vert QR \sqrt{2} \Vert_{\R/\Z} \leq \eta QR$. On the other hand we have the well-known and elementary bound $\Vert m \sqrt{2} \Vert_{\R/\Z} \geq \frac{1}{3m}$ for all positive integers $m$, and thus $Q \geq \frac{1}{3R\sqrt{\eta}}$. Substituting back into \eqref{first} gives
\[ \# \{x : \frac{1}{2\pi i} \log \psi(x) \in J\} \leq 2\eta p +3Rp\sqrt{\eta} + 1,\] and \eqref{edge-bound-3} now follows because of the choice of $\eta$ in \eqref{eta-def} provided $p \geq C\eta^{-1}$.
\end{proof}
\emph{Proof of Claim B:}
We apply the triangle inequality for each $j \in \{1,\dots,d\}$ to get
\begin{equation*}
\big\|\frac{t_j-t_j'}{R}\big\|_{\R/\Z} \leq \big\|\theta_j -\frac{2t_j+1}{2R} - \sqrt{2}\big\|_{\R/\Z}+\big\|\theta_j -\frac{2t_j'+1}{2R} - \sqrt{2}\big\|_{\R/\Z}<\frac{1}{2R}+\eta + \frac{1}{2R}-\eta = \frac{1}{R}.
\end{equation*}
Thus
\begin{equation*}
\min\big\{ \frac{|t_j-t_j'|}{R}, \frac{|R+t_j-t_j'|}{R}, \frac{|t_j-t_j'-R|}{R}\big\} <\frac{1}{R}.
\end{equation*}
Since $-(R-1) \leq t_j-t_j' \leq R-1$ it follows that $t_j=t_j'$ for all $j \in \{1,\dots,d\}$ and similarly for $u$ and $u'$, and $v$ and $v'$.  The claim is proved, and this concludes the proof of Lemma \ref{lem5.4}.
\end{proof}

We are now ready for the proof of the regularity lemma itself.

\begin{proof}[Proof of Proposition \ref{regularity}]
First we have to define $M_\Omega$ and $D_\Omega$. Suppose that $r,d \in \Z_{\geq 0}$, and $\delta \in (0,1]$ and define sequences
\begin{equation*}
\textstyle\frac{1}{10}\displaystyle\delta = \delta_0 > \delta_1 > \dots \text{ and } d = d_0^{\max} < d_1^{\max} < \dots 
\end{equation*}
and auxiliary sequences
\begin{equation*}
R_0 < R_1 < \dots, \qquad  M_0^{\max}< M_1^{\max} < \dots, \qquad \text{ and } \; \; p_0^{\max} < p_1^{\max} < \dots 
\end{equation*}
defined by
\[ R_{j}:=2^{\lceil \log_2C\delta_{j}^{-1}\rceil},  M^{\max}_{j} := M_0(\eps,d_j^{\max},R_j)\text{ and }p_j^{\max}:=p_3(\eps,d_j^{\max},R_j),\]
(where $M_0$ and $p_3$ are the functions appearing in Lemma \ref{lem5.4})
and the following recursive rules:
\[ \delta_{j+1} := \min\{1/\Omega\big(r,d,\delta,M^{\max}_j, d^{\max}_j   \big),\delta_0\} \text{ and }d^{\max}_{j+1} := d^{\max}_j + \lceil Cr \delta_{j+1}^{-2}\rceil.\]

By induction on $j$ and monotonicity of $\Omega$ and $M_0$ we can inductively check that the functions $\delta_j$, $R_j$, $d_j^{\max}$, $M_j^{\max}$, $p_j^{max}$ are monotonic functions of $r,d,\delta,\eps,j$ in the sense introduced at the end of Section \ref{sec3}. Set \begin{equation}\label{J-def} J := \lceil Cr/\eps^2 \rceil\end{equation} and define
\begin{equation*}
M_\Omega(r,d,\delta,\eps):=M_J^{\max}, D_\Omega(r,d,\delta,\eps):=d_J^{\max} \text{ and } p_\Omega(r,d,\delta,\eps):=p_J^{\max}.
\end{equation*}
These are also monotone functions.

With these definitions in hand we are ready for the proof. Suppose that $p \geq p_\Omega(r,d,\delta,\eps)$.  Suppose, further, that $\Psi$ is a $d$-dimensional $\QM$-system of width $\delta$, that $B:=B(\Psi,\delta)$, that $\tilde B \subset B$ has $|\tilde B| \geq (1-\frac{\eps^2}{100})|B|$, and finally that $c:\tilde{B} \rightarrow [r]$ is an $r$-colouring.  We extend $c : \tilde B \rightarrow [r]$ to a full colouring $\overline{c} : B \rightarrow [r]$ in some arbitrary way such that 
\begin{equation}\label{eqn.c}
\Vert 1_{c^{-1}(i)} - 1_{\overline{c}^{-1}(i)} \Vert_2 \leq \frac{\eps}{10} \text{ for all }i \in [r].
\end{equation}
(This could be done, for example, by defining $\overline{c} \equiv 1$ identically on $B \setminus \tilde B$.) 

We shall apply the Koopman von Neumann lemma (Lemma \ref{kvn}) $J$ times (where $J$ is given by \eqref{J-def}), on each occasion with functions $f_i := 1_{\overline{c}^{-1}(i)}$, for $i \in [r]$, extended to functions on all of $\F$ by setting $f_i(x) = 0$ if $x \notin B$. On the $j$th such application we apply Lemma \ref{kvn} with parameter $R_j$, obtaining nested $\QM$-systems $\Psi = :\Psi_0 \subset \Psi_1 \subset \Psi_2 \subset \dots$ with 
\begin{equation}\label{dim-growth}
\dim \Psi_{j} \leq d^{\max}_{j}
\end{equation}
such that 
\begin{equation}\label{eq686}
\Vert f_i - \Pi_{R_j}^{\Psi_j} f_i \Vert_{\QM} \leq \delta_j
\end{equation}
for all $i \in \{ 1,\dots, r\}$. It is convenient to write $P_j := \Pi^{\Psi_j}_{R_j}$ for short; it is also convenient to write $\mathcal{B}_j$ for the $\sigma$-algebra on $\F$ generated by $\{\Psi_j^{-1}(I_{R_j; t, u, v}):t,u,v \in \{0,\dots,R-1\}^d\}$. Thus $P_j = \E( \cdot | \mathcal{B}_j)$. Note that, since $\Psi_j \subset \Psi_{j+1}$ and $R_j | R_{j+1}$ (since $R_j$ and $R_{j+1}$ are powers of $2$), $\mathcal{B}_{j+1}$ is a refinement of $\mathcal{B}_j$ and hence
$P_j = P_j P_{j+1}$. By Pythagoras' theorem we have
\[ \Vert P_{j+1} f_i\Vert_2^2 - \Vert P_j f_i \Vert_2^2 = \Vert P_{j+1} f_i - P_j f_i \Vert_2^2,\] thus
\[ \sum_{j= 1}^J\sum_{i = 1}^r \Vert P_{j+1} f_i - P_j f_i \Vert_2^2 \leq r,\]
so by the pigeonhole principle and the choice of $J$ there is some $j \leq J$ such that 
\begin{equation}\label{eq477} \sum_{i = 1}^r \Vert P_{j+1} f_i - P_j f_i \Vert_2^2 \leq \frac{\eps^2}{100}.\end{equation}

Fix this value of $j$, and set $d_j := \dim \Psi_j$ (thus $d_j \leq d^{\max}_j$). Let $l \in \N$ be minimal such that $2^{-l} \leq \eps$, and apply Lemma \ref{lem5.4} to $f_i$ for each $i \in \{1,\dots,r\}$.  Since
\begin{equation*}
p \geq p_J^{\max} \geq p_j^{\max}= p_3(\eps,d_j^{\max},R_j) \geq p_3(\eps,d_j,R_j),
\end{equation*}
we get functions $F^j_i : \G^{d_j} \rightarrow [0,1]$ such that 
\begin{enumerate}
\item[(i)] $F^j_i \circ \Psi_j \geq P_j f_i$ pointwise;
\item[(ii)] $\Vert F^j_i \Vert_{\trig} \leq M_0(\eps, d_j, R_j) \leq M_j^{\max}$;
\item[(iii)] $\Vert F^j_i \circ \Psi_j - P_j f_i \Vert_2 \leq \eps/2$;
\item[(iv)] $\Vert F^j_i \circ \Psi_j  \Vert_4 \leq 2$.
\end{enumerate}

We are now ready to describe the functions $F_1,\dots, F_r, g_1,\dots, g_r$ in the regularity lemma. Set $d' := d_j$, $\Psi' := \Psi_j$, $F_i:=F_i^j$ for all $i \in [r]$, and
\begin{equation*}
g_i := (P_{j+1} f_i -f_i)+ 1_{c^{-1}(i)} = P_{j+1} f_i + 1_{c^{-1}(i)} - 1_{\overline{c}^{-1}(i)}
\end{equation*}
for all $i \in [r]$.  Note with this choice that the $g_i$s all map into $[-1,1]$ as required. We claim that these functions $F_i$ and $g_i$ do indeed verify (1) to (6) of the regularity lemma; we check these points in turn.

\emph{Point (1).}  This is immediate from (ii) above since
\[\Vert F_i\|_{\trig}=\Vert F_i^j \Vert_{\trig} \leq M_j^{\max} \leq M_J^{\max} =M_\Omega(r,d,\delta,\eps).\]

\emph{Point (2).} This is immediate from \eqref{dim-growth} since that tells us
\[d' \leq d_j^{\max}\leq d_J^{\max} = D_\Omega(r,d,\delta,\eps).\]

\emph{Point (3).} We know from point (iii) above, \eqref{eq477}, and \eqref{eqn.c} above that
\begin{align*}
\Vert F_i\circ \Psi' - g_i\|_2 & = \|F_i^j\circ \Psi_j - g_i\|_2\\
& \leq \|F_i^j\circ \Psi_j -P_jf_i\|_2\\
& \qquad +  \|P_jf_i - P_{j+1}f_i\|_2 + \|1_{c^{-1}(i)} - 1_{\overline{c}^{-1}(i)}\|_2\\
& \leq \frac{\eps}{2} + \sqrt{\frac{\eps^2}{100}} + \frac{\eps}{10}<\eps.
\end{align*}

\emph{Point (4).} We have 
\[ \Vert 1_{c^{-1}(i)} - g_i \Vert_{\QM} = \Vert f_i - P_{j+1} f_i \Vert_{\QM} \leq \delta_{j+1}.\] By the choice of $\delta_{j+1}$, this is at most $1/\Omega(r,d,\delta,M_j^{\max}, d_j^{\max})$.  Since $\|F_i\|_{\trig} \leq M_j^{\max}$ and $d' \leq d_j^{\max}$ and $\Omega$ is monotonic, it follows that this is at most $1/\Omega(r,d,\delta,\|F_i\|_{\trig},d')$ and (4) follows.

\emph{Point (5).} By design we have $\sum_{i=1}^r F_i \circ \Psi'(x) = \sum_{i = 1}^r F_i^j \circ \Psi_j(x)$. On the other hand, by (i) we have $F_i^j \circ \Psi_j \geq P_j f_i$ pointwise, and so summing over $i$ we have
\begin{align*}
\sum_{i=1}^r F_i \circ \Psi'(x) & = \sum_{i = 1}^r F_i^j \circ \Psi_j(x)\\
& \geq \sum_{i = 1}^r P_jf_i(x)\\
& = P_j(f_1 + \dots + f_r)(x)\\
& = P_j(1_{\overline{c}^{-1}(1)} + \dots + 1_{\overline{c}^{-1}(r)})(x) = P_j(1_{B})(x)
\end{align*}
since $P_j$ is linear.   However, since $R_j \geq 2\delta^{-1}$ every atom of $\mathcal{B}_j$ which meets $B(\Psi,\delta/2)$ is entirely contained in $B$, and so $P_j 1_B \geq 1$ pointwise on $B(\Psi,\frac{1}{2}\delta)$.

\emph{Point (6).}  This follows immediately from (iv).

The result is proved.
\end{proof}

\section{Some results in Ramsey theory}\label{sec7}

In this section we state and prove some auxiliary results of a Ramsey-theoretic nature, the main aim being to establish Proposition \ref{ramsey-prop}. A model for the type of result we are interested in is the following: if $\N \times \N$ is finitely coloured, then there is a monochromatic triple of distinct elements $(t_1, u)$, $(t_2, u)$, $(t_3,t_2 - t_1)$. Such a result is certainly true, and can be easily proved as follows.

Given an $r$-colouring of $[N]^2$, consider the colouring induced on $\{(1,1),\dots,(N,1)\}$ which is certainly at most an $r$-colouring.  By van der Waerden's theorem \cite{vdW} there is a monochromatic arithmetic progression
\begin{equation*}
\Delta:=\{(x,1),(x+d,1),\dots,(x+Md,1)\}
\end{equation*}
where $M \rightarrow \infty$ as $N \rightarrow \infty$ (for fixed $r$).  If there is some $t_3$ such that $(t_3,jd)$ has the same colour as $\Delta$ for some $1 \leq j \leq M$ then we have a suitable monochromatic triple, namely $(x,1),(x+jd,1),(t_3,jd)$.  Otherwise the box $[N] \times d[M]$ is $(r-1)$-coloured, and hence so is $(d,d)\cdot [M]^2$.  This forms the basis for an induction on the number of colours, thereby giving the result.

Of course this result is just the tip of the iceberg, and there is a vast generalisation available in recent work of Bergelson, Johnson, and Moreira \cite{bergelson-johnson-moreira}.  Indeed, \cite[Corollary 3.7]{bergelson-johnson-moreira} contains the above as a special case by taking (in the language of that result) $G:=\N_0^2$, $m:=1$, $c:G \rightarrow G$ to be the identity, and letting $F_1$ be the set containing the two maps
\begin{equation*}
\N_0^2 \rightarrow \N_0^2; (x,y) \mapsto (0,0) \text{ and } \N_0^2 \rightarrow \N_0^2; (x,y) \mapsto (y,0).
\end{equation*}
(Verifying that this is the same theorem requires a little work: the set $D(m, \vec{F}, c; \mathbf{s})$ in \cite[Definition 3.1]{bergelson-johnson-moreira} then consists of the elements $(s_{00}, s_{01})$, $(s_{10}, s_{11})$ and $(s_{01} + s_{10}, s_{11})$, which are of the form $(t_1, u)$, $(t_2, u)$, $(t_3,t_2 - t_1)$ with $t_1 = s_{10}$, $t_2 = s_{01} + s_{10}$, $t_3 = s_{00}$, $u = s_{11}$.)

While these extensions are clearly interesting we need to generalise our model in a different direction.  In particular, we need not just one monochromatic triple but many triples.  A result of this type extending Rado's theorem was established in \cite[Theorem 1]{franklgrahamrodl}, and extended to the case of torsion groups in \cite[Theorem 1.3]{serra-vena}.  In particular it is worth noting that \cite[Section 7.1]{serra-vena} identifies some difficulties that emerge in the presence of torsion.

We turn, now, to our arguments.  Let us recall the statement we are aiming to prove. 

\begin{ramsey-lemma}[Ramsey lemma] There is a monotonic function $\rho : \Z_{\geq 0} \times \Z_{\geq 0} \times (0,1] \rightarrow (0,1]$ with the following property. Suppose that $X$ and $Y$ are compact Abelian groups with Haar probability measures $\mu_X, \mu_Y$, that $\pi_X : X \rightarrow (\R/\Z)^{d}$ and $\pi_Y : Y \rightarrow (\R/\Z)^{d}$ are continuous homomorphisms, and that $F_1,\dots, F_r : X \times X \times Y \rightarrow \R_{\geq 0}$ are continuous functions with $\sum_{i=1}^r F_i(x_1, x_2, y) \geq 1$ whenever we have $| \pi_X(x_1)|,|\pi_X(x_2)|,|\pi_Y(y)| \leq \frac{1}{4}\delta$. Let $\mu = \mu_X \times \mu_X \times \mu_Y$. 
Then
\[ {\int F_i(t,u,v) F_i(t + u', u, v') F_i(t', u',v)} d\mu(t,u,v) d\mu(t',u',v') \geq \rho(r,d,\delta)
\] for some $i \in [r]$.
\end{ramsey-lemma}

We shall prove this via the following result.
\begin{proposition}\label{prop.mainramseydriver}
Suppose that $G$ is a compact Abelian group, that $T \subset G$ is open, and that $A_1,\dots,A_r$ are the measurable colour classes of some $r$-colouring on $T \times (T-T)$.  Then
\[ \sum_{i=1}^r \int{1_{A_i}(t_1,t_4-t_5)1_{A_i}(t_2,t_4-t_5)} 1_{A_i}(t_3,t_2-t_1)d\mu_T^{\otimes 5}(t_1,t_2,t_3,t_4,t_5) \geq r^{-O(r)}.
\]
\end{proposition}
Here, $\mu_T := \mu_G(T)^{-1} \mu_G$, where $\mu_G$ is the normalised Haar measure on $G$ which makes sense since $T$ is open and so is measurable and of positive measure.

We have written the bound here explicitly, first to show that it may be taken to be monotonically decreasing in $r$, and secondly because it is not too far from the right order.  Indeed, suppose $G=\F_2^r$ and we have a $(2r+1)$-colouring of the pairs $(t,u) \in G\times (G-G)$ with colour classes defined by
\begin{equation*}
A_0 := \F_2^r \times \{0\} \text{ and }A_i = \{(t,u) : u_1 = \dots = u_{i-1} = 0, u_i = 1, t_i = 0\},
\end{equation*}
and
\begin{equation*}
A_{r + i} = \{(t,u) : u_1 = \dots = u_{i-1} = 0, u_i = 1, t_i = 1\},
\end{equation*}
where in both cases $i$ ranges over $\{1,\dots,r\}$.  There are no triples $(t,u),(t',u),(t'',t+t')$ all lying in $A_i$ for any $i > 0$ since this would imply that $t_i=t_i'$ and so $(t+t')_i=0$, a contradiction.  Hence all the monochromatic triples lie in $A_0$, which implies that their total measure is at most $4^{-r}$.

Before turning to the proof of Proposition \ref{prop.mainramseydriver}, let us see how it implies Proposition \ref{ramsey-prop}.
\begin{proof}[Proof of Proposition \ref{ramsey-prop} assuming Proposition \ref{prop.mainramseydriver}]
Put $f_i:=\min\{F_i(x),1\}$ so that the $f_i$s take values in $[0,1]$, are continuous and have $\sum_i{f_i(x_1,x_2,y)} \geq 1$ whenever $| \pi_X(x_1)|,|\pi_X(x_2)|,|\pi_Y(y)| \leq \delta/4$.  We start by simplifying the dependence on $v$ and $v'$ by noting that since the $f_i$s all take values in $[0,1]$ we have, for each $i \in [r]$,
\begin{align}
\nonumber & \int f_i(t,u,v) f_i(t + u', u, v') f_i(t', u',v) d\mu(t,u,v) d\mu(t',u',v')\\
\nonumber & \geq \int{ f_i(t,u,v) f_i(t + u', u, v') f_i(t', u',v)}\\
\nonumber & \quad \times f_i(t,u,v') f_i(t + u', u, v) f_i(t', u',v') d\mu(t,u,v) d\mu(t',u',v')\\
\nonumber &  = \int{\left(\int{f_i(t,u,v)f_i(t+u',u,v)f_i(t',u',v)d\mu_Y(v)}\right)^2d\mu_{X}^{\otimes 4}(t,u,t',u')}\\ \label{eqn.cssimp} 
 & \geq \left(\int{f_i(t,u,v)f_i(t+u',u,v)f_i(t',u',v)d\mu_{X}^{\otimes 4}(t,u,t',u')d\mu_Y(v)}\right)^2,
\end{align}
the last step here being a consequence of the Cauchy-Schwarz inequality.
Let $B \subset X \times X$ be the set of pairs $(t,u)$ such that $|\pi_X(t)|, |\pi_X(u)| \leq \frac{1}{4}\delta$. For each $v \in Y$, define
\begin{equation*}
A_i^{(v)}:=\big\{(t,u)\in B: f_i(t,u,v) \geq 1/r \text{ and } (t,u) \not \in \bigcup_{j=1}^{i-1}{A_j^{(v)}}\big\}.
\end{equation*} 
By hypothesis we have
\begin{equation*}
\sum_i{f_i(t,u,v)} \geq 1 \text{ whenever } | \pi_X(t)|,|\pi_X(u)|,|\pi_Y(v)| \leq \textstyle\frac{1}{4}\displaystyle\delta
\end{equation*}
and so by averaging
\begin{equation*}
\bigcup_{i=1}^r{A_i^{(v)}} = B \text{ whenever } |\pi_Y(v)| \leq \textstyle\frac{1}{4}\displaystyle\delta.
\end{equation*}
Moreover since the $f_i$s and $\pi_X$ are continuous the sets $A_i^{(v)}$ are measurable. 
We claim (and shall prove later) that for any choice of measurable sets $(A_i)_{i=1}^r$ with $\bigcup_i{A_i} = B$ we have
\begin{equation}\label{77claim}
\sum_{i=1}^r{\int{1_{A_i}(t,u)1_{A_i}(t+u',u)1_{A_i}(t',u')d\mu_{X}^{\otimes 4}(t,u,t',u')}} \geq r^{-O(r)}(\delta/8)^{5d}.
\end{equation}
Assuming this for now, by \eqref{eqn.cssimp} and the design of the sets $A_i^{(v)}$ we have
\begin{align*}
 \sup_i&\int f_i(t,u,v) f_i(t + u', u, v') f_i(t', u',v) d\mu(t,u,v) d\mu(t',u',v') \\
& \geq \bigg(\sup_i \int{f_i(t,u,v)f_i(t+u',u,v)f_i(t',u',v)}d\mu_{X}^{\otimes 4}(t,u,t',u')d\mu_Y(v)\bigg)^2\\
& \geq \bigg(\frac{1}{r^4}\int {\left(\sum_{i=1}^r{\int{1_{A_i^{(v)}}(t,u)1_{A_i^{(v)}}(t+u',u)} 1_{A_i^{(v)}}(t',u')}d\mu_{X}^{\otimes 4}(t,u,t',u')\right)d\mu_Y(v)}\bigg)^2.
\end{align*}
By \eqref{77claim} and the comments before it, the inner integral is bounded below by $r^{-O(r)}(\delta/8)^{5d}$ uniformly for $v$ with $|\pi_Y(v)| \leq \frac{1}{4}\delta$. By Lemma \ref{pigeon-projection} we have
\[ \int_{|\pi_Y(v)| \leq \delta/4} d\mu_Y(v) \geq (\delta/4)^d.\]
Putting these facts together we get that
\begin{align*}
 \sup_i&\int F_i(t,u,v) F_i(t + u', u, v') F_i(t', u',v) d\mu(t,u,v) d\mu(t',u',v')\\
 & \geq \sup_i\int f_i(t,u,v) f_i(t + u', u, v') f_i(t', u',v) d\mu(t,u,v) d\mu(t',u',v') \\
& \geq \bigg(\sup_i \int{f_i(t,u,v)f_i(t+u',u,v)f_i(t',u',v)}d\mu_{X}^{\otimes 4}(t,u,t',u')d\mu_Y(v)\bigg)^2\\
& \geq \bigg(\frac{1}{r^4}\cdot r^{-O(r)}(\delta/8)^{5d}\cdot (\delta/4)^d\bigg)^2.\end{align*}
This tells us that we may define $\rho$ monotonically in the appropriate way, and concludes the proof of Proposition \ref{ramsey-prop} assuming Proposition \ref{prop.mainramseydriver}, except for the need to establish the claim \eqref{77claim}. We turn now to the proof of this claim. First of all we introduce a dummy integration over $X$, so the claim to be established becomes
\[ \sum_{i = 1}^r \int 1_{A_i}(t,u) 1_{A_i}(t + u', u) 1_{A_i}(t', u') d\mu_X^{\otimes 5}(t,u,t,u',v) \geq r^{-O(r)}(\delta/8)^{5d}.\]
Now make the change of variables $t = t_1$, $u = t_4 - t_5$, $t' = t_3$, $u' = t_2 - t_1$, $v = t_5$. This change of variables is invertible (by $t_1 = t$, $t_2 = t + u'$, $t_3 = t'$, $t_4 = u + v$, $t_5 = v$) and preserves the Haar measure, so it is enough to show that 
\[ \sum_{i = 1}^r \int 1_{A_i}(t_1, t_4 - t_5) 1_{A_i}(t_2, t_4 - t_5) 1_{A_i}(t_3, t_2 - t_1) d\mu_X^{\otimes 5}(t_1, t_2, t_3, t_4, t_5)\geq r^{-O(r)}(\delta/8)^{5d}.\]
Now we are told that $\bigcup_i A_i = B = \{(x,x') : |\pi_X(x)|, |\pi_X(x')| \leq \frac{1}{4}\delta\}$. It follows that the $A_i$ restrict to give an $r$-colouring of $T \times (T - T)$, where $T := \{x \in X : |\pi_X(x)| \leq \frac{1}{8}\delta\}$. Since $\mu_X(T) \geq (\delta/8)^d$ by Lemma \ref{pigeon-projection}, the claim now follows from Proposition \ref{prop.mainramseydriver}.
\end{proof}

The remaining task of the section, then, is to establish Proposition \ref{prop.mainramseydriver}. A key ingredient of this is a ``dependent random selection'' result of the type pioneered by Gowers \cite{gowers-4ap}.  It is given in almost this specific form as \cite[Lemma 6.17]{taovu} the proof of which is itself contained in \cite[Lemma 4.2]{sudszevu}. We need a weighted version of their result so we provide a self-contained proof, though the argument is more-or-less identical.
\begin{lemma}\label{lem.gowersweakregularity}
Suppose that $(X,\nu_X)$ and $(Y,\nu_Y)$ are probability spaces, $A \subset X \times Y$ is measurable with $(\nu_X \times \nu_Y)(A) = \alpha$, and let $\eta \in (0,1]$ be a parameter.  Then there is a measurable set $X' \subset X$ with $\nu_X(X') \geq \frac{1}{2}\alpha$ such that the set
\begin{equation*}
E:=\big\{(x_1,x_2) \in X \times X: \nu_Y\left(\left\{ y \in Y : (x_1, y), (x_2, y) \in A\right\}\right) \leq \textstyle\frac{1}{2}\displaystyle\eta \alpha^2\big\}
\end{equation*}
\textup{(}which is visibly measurable\textup{)} has $\nu_X^{\otimes 2}(E \cap (X' \times X')) \leq \eta \nu_X(X')^2$.
\end{lemma}
In words, a (weighted) proportion $1 - \eta$ of the ``edges'' in $X'$ have many common neighbours in $Y$.
\begin{proof}
For $x \in X$, write $N_Y(x) := \{y \in Y: (x,y) \in A\}$ and for $y \in Y$ write $N_X(y) := \{x \in X : (x,y) \in A\}$. By Fubini's theorem we have
\[ \int \nu_Y(N_Y(x_1) \cap N_Y(x_2)) d\nu_{X}^{\otimes 2}(x_1,x_2) = \int \nu_X(N_X(y))^2 d\nu_Y(y),\]
with both sides being equal to
\[ \int 1_{A}(x_1,y)1_A(x_2,y)d\nu_X^{\otimes 2}(x_1,x_2) d\nu_Y(y).\]
By the Cauchy-Schwarz inequality and the identity
\[ \int \nu_X(N_X(y)) d\nu_Y(y) = \iint 1_{A}(x,y) d\nu_X(x) d\nu_Y(y) = \alpha,\] we see that 
\begin{equation}\label{eq421}
\int \nu_Y(N_Y(x_1) \cap N_Y(x_2)) d\nu_X^{\otimes 2}(x_1, x_2) \geq \alpha^2.
\end{equation}
By definition of $E$ we have
\[ \iint 1_E(x_1,x_2) \nu_Y(N_Y(x_1) \cap N_Y(x_2)) d\nu_X^{\otimes 2}(x_1, x_2) \leq \textstyle\frac{1}{2}\displaystyle\eta \alpha^2.\]
Putting this together with \eqref{eq421}, we see that 
\[ \int \big(1 - \frac{1}{\eta} 1_E(x_1,x_2)\big) \nu_Y(N_Y(x_1) \cap N_Y(x_2)) d\nu_X^{\otimes 2}(x_1, x_2) \geq \textstyle\frac{1}{2}\displaystyle\alpha^2, \] or in other words (by Fubini's theorem)
\[ \iint \big(1 - \frac{1}{\eta} 1_E(x_1,x_2) \big) 1_{N_X(y)}(x_1)1_{N_X(y)}(x_2) d\nu_X^{\otimes 2}(x_1, x_2) d\nu_Y(y) \geq \textstyle\frac{1}{2}\displaystyle\alpha^2.\]
In particular, there is some specific choice of $y$ for which ($N_X(y)$ is measurable and)
\[ \int \big(1 - \frac{1}{\eta} 1_E(x_1,x_2) \big) 1_{N_X(y)}(x_1)1_{N_X(y)}(x_2)  d\nu_X^{\otimes 2}(x_1, x_2)  \geq \textstyle\frac{1}{2}\displaystyle\alpha^2.\]
For this $y$ set $X' := N_X(y)$, and we have both 
\[ \int 1_{X'}(x_1)1_{X'}(x_2) d\nu_X^{\otimes 2}(x_1, x_2) \geq \textstyle\frac{1}{2}\displaystyle\alpha^2,\] whence $\nu_X(X') \geq \frac{1}{2}\alpha$, and
\[ \int 1_{X'}(x_1)1_{X'}(x_2) d\nu_X^{\otimes 2}(x_1, x_2) \geq  \frac{1}{\eta} \int 1_E(x_1,x_2) 1_{X'}(x_1)1_{X'}(x_2) d\nu_X^{\otimes 2}(x_1, x_2),\] or in other words $\nu_X^{\otimes 2} (E \cap (X' \times X')) \leq \eta \nu_X^{\otimes 2}(X')^2$.  The result is proved.
\end{proof}

In what follows we shall be working in products $T \times (T - T)$, where $T \subset G$.  If $A \subset G \times G$ is measurable, we define
\begin{equation*}
\delta_T(A) := \int{1_{A}(t,t_1-t_2)d\mu_T(t)d\mu_T(t_1)d\mu_T(t_2)},
\end{equation*}
and
\[
\Lambda_T(A) := \int{1_A(t_1,t_4-t_5)1_A(t_2,t_4-t_5)1_A(t_3,t_2-t_1)}d\mu_T^{\otimes 5}(t_1,t_2,t_3,t_4,t_5).
\]
In this notation, Proposition \ref{prop.mainramseydriver} may be restated as follows: if $c : T \times (T - T) \rightarrow [r]$ is a colouring then there is some $i \in [r]$ for which $\Lambda_T(c^{-1}(i)) \geq r^{-O(r)}$. We shall establish this by induction on the number of colours (one of two places in our paper where we do this, the other being in the proof of Proposition \ref{main-prop-refined}). To carry  this out we need to prove a slightly stronger statement.

\begin{proposition}
Set $\eps_r := 2^{-7r+1} (r!)^{-3}$. Suppose that $G$ is a compact Abelian group, $T \subset G$ is an open set, $E \subset T \times (T-T)$ is measurable and $c : T \times (T - T) \rightarrow [r]$ is a measurable partial colouring, defined outside of a set $E$ with $\delta_T(E) \leq \eps_r$. Then there is $i \in [r]$ with $\Lambda_T(c^{-1}(i)) \geq \eps_r^2$.
\end{proposition}
\begin{proof}
We proceed by induction on $r$, the result being vacuously true when $r =  0$ since $\eps_0 = \frac{1}{2}$ and there are no $0$-colourings of non-empty sets. Suppose we know the result for $r-1$, and that we have a partial $r$-colouring of $T \times (T - T)$ as described.

Certainly $\eps_r < \frac{1}{2}$, so by the pigeonhole principle there is some $i$ for which $\delta_T(A_i) \geq 1/2r$ where $A_1,\dots,A_r$ are the colour classes of $c$. We shall apply Lemma \ref{lem.gowersweakregularity} with $X = T$, $Y = T - T$, $A = c^{-1}(i)$ and with $\eta = \frac{1}{4}\eps_{r-1}$. Let $\mu_X$ be the probability measure induced on $T$ by the Haar probability measure $\mu_G$ and let the measure on $Y$ be given by $\mu_Y:=\mu_T \ast \mu_{-T}$. The lemma outputs a measurable set $T'$ (that is, $X'$ in the lemma) with $\mu_T(T') \geq 1/4r$ and and a measurable set $Z \subset T' \times T'$ consisting of a proportion at least $1 - \frac{1}{4}\eps_{r-1}$ of all pairs $(t_1, t_2)$ in $T' \times T'$ such that
\begin{equation}\label{882}
\int{1_{A_i}(t_1,t_4-t_5)1_{A_i}(t_2,t_4-t_5)d\mu_T(t_4)d\mu_T(t_5)}\geq \frac{\eta}{8r^2}
\end{equation}
whenever $(t_1, t_2) \in Z$.

Suppose in the first instance that $\delta_{T'}(A_i) \geq \frac{1}{2}\eps_{r-1}$.  Then
\begin{align*}
\Lambda_T(A_i) &\geq \mu_T(T')^3\int{1_{A_i}(t_3,t_2-t_1)\int{1_{A_i}(t_1,t_4-t_5)}1_{A_i}(t_2,t_4-t_5)}d\mu_T(t_4)d\mu_T(t_5)d\mu_{T'}^{\otimes 3}(t_1,t_2,t_3)\\
& \geq \frac{\eta\mu_T(T')^3}{8r^2}\int{1_{A_i}(t_3,t_2-t_1)1_Z(t_1,t_2)}d\mu_{T'}^{\otimes 3}(t_1,t_2,t_3)\\
& \geq \frac{\eta\mu_T(T')^3}{8r^2}\bigg(\int{1_{A_i}(t_3,t_2-t_1)d\mu_{T'}^{\otimes 3}(t_1,t_2,t_3)} - \textstyle\frac{1}{4}\displaystyle\eps_{r-1} \bigg)\\
& \geq \frac{\eta\mu_T(T')^3}{8r^2}\cdot \frac{\epsilon_{r-1}}{4} \geq 2^{-13} r^{-5} \eps^2_{r-1} > \eps_r^2.
\end{align*}
The alternative is that $\delta_{T'}(A_i) \leq \frac{1}{2}\eps_{r-1}$. Noting that
\[ \delta_{T'}(E) \leq \mu_T(T')^{-3} \delta_T(E) \leq 64 r^3 \eps_r = \textstyle\frac{1}{2}\displaystyle\eps_{r-1},\] it follows that $c$ restricts to give a partial colouring $c' : T' \times (T' - T') \rightarrow [r] \setminus \{i\}$, defined outside of a set $E' = (E \cup c^{-1}(i)) \cap (T' \times (T' - T'))$ with $\delta_{T'}(E') \leq \eps_{r-1}$. By our inductive hypothesis, there is $j$ such that $\Lambda_{T'}(A_j) \geq \eps_{r-1}^2$, which implies that
\[\Lambda_T(A_j) \geq \mu_T(T')^5 \Lambda_{T'}(A_j) \geq (4r)^{-5}\eps_{r-1}^2 \geq \eps_r^2.\]
This concludes the proof. 
\end{proof}

\section{The baby counting lemma and the counting lemma}\label{sec8}

The objective of this section is to prove the baby counting lemma from \S \ref{sec3} and the counting lemma from \S\ref{sec4}, the first of these acting as a kind of warm up to the second. 

The proofs require various lemmas on exponential sums with characters, and we begin by assembling these. There is a great deal of general theory on this topic; see, for example,  \cite[Chapter 11]{iwaniec-kowalski}. We develop just what we need for our application, namely Proposition \ref{add-mult} below, eschewing any temptation to seek the strongest available bounds. In fact, a bound of $o(1)$ times the trivial bound is all we need.

Shkredov \cite{shkredov} makes use of the following result from Johnsen \cite{johnsen} which he (Shkredov) records as \cite[Theorem 2.4]{shkredov}.
\begin{theorem}[{\cite[Lemma 1]{johnsen}}]\label{char-theorem}
Suppose that $\K$ is a finite field, $\chi_1,\dots,\chi_t$ are $t<|\K|$ multiplicative characters on $\K$ with $\chi_i$ non-principal for some $i$, and $h_1,\dots,h_t$ are distinct elements of $\K$.  Then
\begin{equation*}
\big|\sum_{x \in \K}{\chi_1(x+h_1)\dots\chi_t(x+h_t)}\big| \leq (t-1)\sqrt{|\K|}
\end{equation*}
\end{theorem}
We shall also use this result but could equally take \emph{e.g.} \cite[Theorem 2C', p.43]{schmidt} or \cite[Theorem 11.23]{iwaniec-kowalski} instead.

\emph{Remark.} Note, of course, that \emph{our} multiplicative characters are extended to the whole of $\F$ in a slightly different way to usual since they are $1$ at $0$.  This can only add an error of size $t$ in the sum in Theorem \ref{char-theorem} which will be of no consequence to us.

With this in hand we are in a position to prove our key proposition.
\begin{proposition}\label{add-mult}
Suppose that $q : \F \rightarrow \F; x \mapsto ax^2 + bx$, $\chi, \chi' : \F \rightarrow \C^*$ are multiplicative characters, and $h \in \F^*$. Then we have
\[ \E_x e_p(q(x)) \chi(x) \chi'(x + h)=o_{p \rightarrow \infty}(1),\]
uniformly in $a,b,\chi,\chi'$, unless $a = b = 0$ and $\chi = \chi' = 1$.
\end{proposition}
\begin{proof}
Suppose first that $\chi = \chi' = 1$. Then the sum reduces to 
\[ \E_x e_p(ax^2 + bx),\] and it follows immediately from the standard Gauss sum estimate that this is bounded by $O(p^{-1/2})$ unless $a = 0$; if $a=0$ then the sum is $0$ by orthogonality of characters unless $b = 0$. Suppose, then, that we do not have $\chi = \chi' = 1$. Writing $G(x):=e_p(-q(x))$ and $F(x):=\chi(x) \chi'(x + h)$ we have
\[ \big|\E_x e_p(q(x)) \chi(x) \chi'(x + h)\big|  = \big|\E_{x,z_1,z_2,z_3}{F(x)\!\!\!\!\prod_{ \omega \in \{0,1\}^3 \setminus (0,0,0)}{\mathcal{C}^{|\omega|} G(x+\omega\cdot z)}}\big| \leq \|F\|_{U^3}
\]
by the Gowers-Cauchy-Schwarz inequality (see, for example, \cite[Equation (11.6)]{taovu}; here $\Vert \cdot \Vert_{U^3}$ is the Gowers $U^3$-norm on $\F$, discussed in many places including \cite[Chapter 11]{taovu}, and $\mathcal{C}$ is the complex conjugation operator). Thus it suffices to establish that $\Vert F \Vert_{U^3} = o(1)$, uniformly in $\chi, \chi'$ and $h$; we shall in fact establish the stronger bound $\|F\|_{U^3}=O(p^{-1/16})$. In other words, we shall prove that 
\begin{align}\label{to-prove-7}
&\E_{x, z_1, z_2, z_3} \prod_{\omega \in \{0,1\}^3} \mathcal{C}^{|\omega|} \chi(x + \omega\cdot z) \chi'(x + h + \omega\cdot z) =O(p^{-1/2}),
 \end{align}
 where $\omega\cdot z:=\omega_1z_1+\omega_2z_2+\omega_3z_3$.  Since $h \neq 0$, there are $O(p^2)$ triples $z \in \F^3$ such that there are some $\omega,\omega' \in \{0,1\}^3$ with $h+\omega\cdot z=\omega'\cdot z$.  For all other triples we may apply Theorem \ref{char-theorem} to see that
 \begin{align*}
&\E_{x} \prod_{\omega \in \{0,1\}^3} \mathcal{C}^{|\omega|} \chi(x + \omega\cdot z) \chi'(x + h + \omega\cdot z) =O(p^{-1/2}).
 \end{align*}
 Equation (\ref{to-prove-7}) follows immediately, and so does the proposition.
\end{proof}

This concludes our discussion of character sum estimates. We turn now to the proofs of the counting lemma and baby counting lemma. Before proceeding it may help the reader to recall some of the definitions from \S\ref{sec3}, particularly the definition of the trig-norm, Definition \ref{deftrig}. We begin by recalling some basic facts about duality.

\begin{lemma}\label{duality-fact}
Suppose that $\Lambda \subset \Z^d$ is a lattice of full rank. Write $G^+ := \{g \in (\R/\Z)^d: \xi \cdot g = 0 \; \mbox{for all $\xi \in \Lambda$}\}$. Suppose that $\lambda \in \Z^d$ satisfies $\lambda \cdot g = 0$ for all $g \in G^+$. Then $\lambda \in \Lambda$.
\end{lemma}
\begin{proof}
Suppose that $\lambda \notin \Lambda$. Let $e_1,\dots, e_d$ be an integral basis for $\Lambda$. Suppose that $\lambda = \lambda_1 e_1 + \dots + \lambda_d e_d$ where, without loss of generality, $\lambda_1 \notin \Z$. By linear algebra there is some $x \in \R^d$ with $e_1 \cdot x = 1$ and $e_2 \cdot x = \dots = e_d \cdot x = 0$. In particular $\xi \cdot x \in \Z$ for $\xi \in \Lambda$, but $\lambda \cdot x = \lambda_1 \notin \Z$. Taking $g = \pi(x)$, where $\pi : \R^d \rightarrow (\R/\Z)^d$ is the natural projection, it follows that $g \in G^+$ but that $\lambda \cdot g \neq 0$ in $(\R/\Z)^d$. \end{proof}

(In fact, the full rank hypothesis is unnecessary, but it is satisfied in our applications.)
An easy consequence of this is the following.

\begin{corollary}\label{mult-cor}
Suppose that $\Lambda \subset \Z^d$ is a lattice of full rank. Write $G^{\times} := \{z \in (S^1)^d: z^{\xi} = 1 \; \mbox{for all $\xi \in \Lambda$}\}$. Suppose that $\lambda \in \Z^d$ satisfies $z^{\lambda} = 1$ for all $z \in G^{\times}$. Then $\lambda \in \Lambda$.
\end{corollary}
\begin{proof}
This follows immediately from Lemma \ref{duality-fact} using the isomorphism $\pi : (\R/\Z)^d \rightarrow (S^1)^d$ defined by $\pi(\theta_1,\dots, \theta_d) = (e(\theta_1),\dots, e(\theta_d))$. 
\end{proof}

Using these facts, we may establish the following key orthogonality relations.

\begin{lemma} We have the orthogonality relations
\begin{equation}\label{orth-plus}
\int e(\xi \cdot t) d\mu_{G_{\Psi}^+}(t) = 1_{\Lambda_{\Psi}^+} \end{equation}
and
\begin{equation}\label{orth-times}
 \int v^{\xi} d\mu_{G_{\Psi}^{\times}}(v) = 1_{\Lambda_{\Psi}^{\times}}.
\end{equation}
\end{lemma}
\begin{proof}
As noted in \S\ref{sec3}, $\Lambda_{\Psi}^+, \Lambda_{\Psi}^{\times}$ both have full rank. We begin with \eqref{orth-plus}. It is clear that if $\xi \in \Lambda_{\Psi}^+$ then the integral is $1$. Conversely, suppose that the integral is nonzero. Let $g \in G_{\Psi}^+$, and make the substitution $t = t' + g$, which preserves the Haar measure. We obtain
\[ \int e(\xi \cdot t) d\mu_{G_{\Psi}^+}(t) = e(\xi \cdot g) \int e(\xi \cdot t') d\mu_{G_{\Psi}^+}(t'),\] and so $e(\xi \cdot g) = 1$, which implies that $\xi \cdot g = 0$ in $(\R/\Z)^d$. Since $g$ was arbitrary, it follows from Lemma \ref{duality-fact} and the definition of $G_{\Psi}^+$ that $\xi \in \Lambda_{\Psi}^+$. The proof of \eqref{orth-times} is extremely similar, using Corollary \ref{mult-cor} in place of Lemma \ref{duality-fact}. 
\end{proof}

Now we turn to the main business of the section, beginning with the baby counting lemma, the statement of which was as follows. The proof is quite straightforward.

\begin{distribution-repeat}
Let $\Psi$ be a $\QM$-system of dimension $d$, and let $F : \G^d \rightarrow \C$ be a trigonometric polynomial. Then we have
\[ \E_{x \in \F} F(\Psi(x)) = \int F d\mu_{H_{\Psi}} + o_{p \rightarrow \infty}(\Vert F \Vert_{\trig}).\]
\end{distribution-repeat}
\begin{proof}
Suppose that $\Vert F \Vert_{\trig} = M$. Expand $F$ as \begin{equation}\label{F-fourier} F(\theta_1, \theta_2, z) = \sum_{\xi_1,\xi_2,\xi_3} \hat{F}(\xi_1, \xi_2,\xi_3) e(\xi_1 \cdot \theta_1 + \xi_2 \cdot \theta_2) z^{\xi_3},\end{equation} where $\sum |\hat{F}(\xi_1, \xi_2, \xi_3)| \leq M$. Suppose that we write $\Psi(x) = ((a_i x^2/p, 2a_i x/p, \psi_i(x))_{i = 1}^d)$, and set $a := (a_1,\dots, a_d)$. 
Then we have
\[ \E_{x \in \F} F(\Psi(x)) = \sum_{\xi_1,\xi_2, \xi_3}\hat{F}(\xi_1, \xi_2,\xi_3) \E_{x \in \F} e_p(\xi_1 \cdot a x^2 + \xi_2 \cdot 2a x) \psi^{\xi_3}(x),\]
where $\psi^{\xi_3}(x):=\prod_i{\psi_i^{(\xi_3)_i}(x)}$ which is a multiplicative character. By Proposition \ref{add-mult} (with $\chi' = 1$) the inner average is $o_{p \rightarrow \infty}(1)$ unless $\xi_1 \cdot a=0$, $\xi_2\cdot 2a=0$, and $\psi^{\xi_3}=1$.  It follows from the definitions that $\xi_1, \xi_2 \in \Lambda_{\Psi}^+$ and $\xi_3 \in \Lambda_{\Psi}^{\times}$, and in that case the inner average equals 1. Since $\sum_{\xi_1,\xi_2,\xi_3} |\hat{F}(\xi_1, \xi_2,\xi_3)| \leq M$, the contribution from those $\xi_1,\xi_2,\xi_3$ not satisfying these conditions is $o_{p \rightarrow \infty}(M)$ and hence
\[ \E_{x \in \F} F(\Psi(x)) = \sum_{\xi_1, \xi_2 \in \Lambda_{\Psi}^+, \xi_3 \in \Lambda_{\Psi}^{\times}} \hat{F}(\xi_1,\xi_2,\xi_3) + o_{p \rightarrow \infty}(M).\]
We claim that 
\begin{equation}\label{8-claim} \sum_{\xi_1, \xi_2 \in \Lambda_{\Psi}^+, \xi_3 \in \Lambda_{\Psi}^{\times}} \hat{F}(\xi_1,\xi_2,\xi_3) = \int F d\mu_{H_{\Psi}},\end{equation} which is enough to conclude the proof. To prove the claim, replace $F$ by its Fourier expansion \eqref{F-fourier} on the right hand side, and recall that $\mu_{H_{\Psi}} = \mu_{G_{\Psi}^+} \times \mu_{G_{\Psi}^+} \times \mu_{G_{\Psi}^{\times}}$. The right hand side is then
\[ \int \sum_{\xi_1,\xi_2 ,\xi_3} \hat{F}(\xi_1,\xi_2,\xi_3) e(\xi_1 \cdot \theta_1)e(\xi_2 \cdot \theta_2) z^{\xi_3} d\mu_{G_{\Psi}^+}(\theta_1) d\mu_{G_{\Psi}^+}(\theta_2) d\mu_{G_{\Psi}^{\times}}(z) .\] By the orthogonality relations \eqref{orth-plus} and \eqref{orth-times}, this is precisely the left-hand side of \eqref{8-claim}.
 \end{proof}

\emph{Remark.} We did not make any use of the fact that $\hat{F}$ was supported where $\|\xi_1\|_1,\|\xi_2\|_1,\|\xi_3\|_1 \leq M$ in this argument, but this will be important in the next argument.

We turn now to the counting lemma itself, the proof of which is related to \cite[Theorem 1.2]{shkredov}.

\begin{counting-lem-repeat}[Counting lemma]
Suppose $\Psi$ is a $d$-dimensional $\QM$-system, $F : \G^d \rightarrow \C$ is a trigonometric polynomial, and that $S \subset B(\Psi,\eps)$.  Then \begin{align*} & T(F \circ \Psi, 1_S, F \circ \Psi, F \circ \Psi) \\ & \qquad =   \mu_{\F}(S) \int F(t,u,v) F(t + u', u, v') F(t', u', v) d\mu_{H_{\Psi}}(t,u,v) d\mu_{H_{\Psi}}(t',u',v') \\ & \qquad  \qquad +   O(\epsilon \mu_{\F}(S)\Vert F \Vert_{\trig}^4) + o_{p \rightarrow \infty}(\Vert F \Vert_{\trig}^{9d}).
\end{align*}
\end{counting-lem-repeat}
\begin{proof}
As before, write $M = \Vert F \Vert_{\trig}$ and expand $F$ as
\[ F(\theta_1, \theta_2, z) = \sum_{\|\xi_1\|_1,\|\xi_2\|_1, \|\xi_3\|_1 \leq M} \hat{F}(\xi_1, \xi_2,\xi_3) e(\xi_1 \cdot \theta_1 + \xi_2 \cdot \theta_2) z^{\xi_3}\] with $\sum |\hat{F}(\xi_1,\xi_2,\xi_3)| \leq M$.
We define the function
\begin{align*}
E(y) :&= \E_{x \in \F} F(\Psi(x)) F(\Psi(x + y)) F(\Psi(xy))\\
& =  \sum_{\substack{\xi_1,\dots, \xi_9 \in \Z^d \\ \|\xi_i\|_1 \leq M}} \nonumber \hat{F}(\xi_1,\xi_2,\xi_3) \hat{F}(\xi_4,\xi_5,\xi_6) \hat{F}(\xi_7, \xi_8, \xi_9)\\
& \qquad \qquad \times \E_{x \in \F} \bigg( e_p\big(\xi_1 \cdot ax^2 + \xi_2 \cdot 2ax + \xi_4 \cdot a(x+y)^2\\
& \qquad \qquad \qquad \qquad  + \xi_5 \cdot 2a(x + y) + \xi_7 \cdot a x^2 y^2 + \xi_8 \cdot 2ax y\big) \\ & \qquad \qquad \qquad \qquad \qquad  \times \psi(x)^{\xi_3} \psi(x + y)^{\xi_6}  \psi(xy)^{\xi_9}\bigg),
\end{align*}
where we write $\psi(x)^{\xi_i}=\prod_j{\psi_j(x)^{(\xi_i)_j}}$ etc.  Now, $\psi_i(xy)=\psi_i(x)\psi_i(y)$ unless $x=0$ or $y=0$, so writing
\begin{align*}
S(y) := \E_{x \in \F} 
& e_p\big((\xi_1 + \xi_4) \cdot a x^2 +\xi_7 \cdot ax^2y^2\\
 & \qquad + (\xi_2 + \xi_5) \cdot 2a x + (\xi_4+\xi_8)\cdot 2axy\big)\\
 & \qquad \qquad \times \psi(x)^{\xi_3+\xi_9} \psi(x + y)^{\xi_6},
 \end{align*}
if $y\neq 0$ we get
 \begin{align}
\label{eqn.bg} E(y) & =  \sum_{\substack{\xi_1,\dots, \xi_9 \in \Z^d \\ \|\xi_i\|_1 \leq M}} \nonumber \hat{F}(\xi_1,\xi_2,\xi_3) \hat{F}(\xi_4,\xi_5,\xi_6) \hat{F}(\xi_7, \xi_8, \xi_9) S(y)e_p(\xi_4 \cdot ay^2 + \xi_5 \cdot 2ay) \psi(y)^{\xi_9}\\ \nonumber & \qquad \qquad \qquad  + O(p^{-1}M^3),
\end{align}
since
\begin{equation*}
\sum_{\substack{\xi_1,\dots, \xi_9 \in \Z^d \\ \|\xi_i\|_1 \leq M}} |\hat{F}(\xi_1,\xi_2,\xi_3) \hat{F}(\xi_4,\xi_5,\xi_6) \hat{F}(\xi_7, \xi_8, \xi_9)| \leq M^3.
\end{equation*}
We are only interested in $E(y)$ when $y \in B(\Psi,\eps)$. In this case, since $\|\xi_4\|_1,\|\xi_5\|_1,\|\xi_9\|_1 \leq M$ we have
\begin{align*}
|e_p  (\xi_4 & \cdot ay^2  + \xi_5 \cdot 2ay) \psi(y)^{\xi_9}-1| \\ & \leq \|\xi_4\|_1\sup_i{|e_p(a_iy^2)-1|} + \|\xi_5\|_1\sup_i{|e_p(2a_iy)-1|} + \|\xi_9\|_1\sup_i{|\psi_i(y)-1|}  = O(\eps M).
\end{align*}
It follows that for $y \in B(\Psi,\eps)\setminus \{0\}$ we have
\begin{equation}\label{e-s}
E(y)  =  \sum_{\substack{\xi_1,\dots, \xi_9 \in \Z^d \\ \|\xi_i\|_1 \leq M}} \hat{F}(\xi_1,\xi_2,\xi_3) \hat{F}(\xi_4,\xi_5,\xi_6) \hat{F}(\xi_7, \xi_8, \xi_9) S(y) +O(\eps M^4) + O(p^{-1}M^3).
\end{equation}
We may re-write $S$ as
 \[ S(y) = \E_{x \in \F} 
 e_p\bigg((\xi_1 + \xi_4 + \xi_7 \overline{y^2}) \cdot a x^2  + \big((\xi_2 + \xi_5) + (\xi_4 + \xi_8)\overline{y}\big)\cdot 2a x\bigg)   \psi(x)^{\xi_3+\xi_9} \psi(x + y)^{\xi_6},\] where $\overline{t}$ denotes the lift of $t \in \F$ to $\{0,\dots, p-1\}$. Now by Proposition \ref{add-mult}, since $y \in \F^*$, $S(y)$ is $o_{p \rightarrow \infty}(1)$ unless 
\begin{equation}\label{add-const} \xi_1 + \xi_4 + \xi_7 \overline{y^2} , \xi_2 + \xi_5 + \overline{y}(\xi_4 + \xi_8) \in \Lambda_{\Psi}^+ \end{equation} and
\begin{equation}\label{mult-const} \xi_3 + \xi_9, \xi_6 \in \Lambda_{\Psi}^{\times}.\end{equation}
The reader may again wish to take a moment and recall the definitions of $\Lambda_{\Psi}^+$ and $\Lambda_{\Psi}^\times$, given at the start of \S\ref{sec3}. It follows from this and the bound $\Vert F \Vert_{\trig} \leq M$ that for $y \in B(\Psi,\eps)\setminus \{0\}$ we have
\begin{align}
\nonumber E(y)  & = o_{p \rightarrow \infty}(M^3) + O(\eps M^4) +\\
& \qquad \qquad
\sum_{\substack{\xi_1,\dots, \xi_9 : \|\xi_i\|_1 \leq M \\ \xi_1 + \xi_4 +\xi_7\overline{y^2}, \xi_2 + \xi_5+\overline{y}(\xi_4 + \xi_8) \in \Lambda_{\Psi}^+ \\ \xi_3 + \xi_9, \xi_6 \in \Lambda_{\Psi}^{\times}}}\!\!\!\!\!\! \hat{F}(\xi_1,\xi_2,\xi_3) \hat{F}(\xi_4,\xi_5,\xi_6) \hat{F}(\xi_7, \xi_8, \xi_9).\label{eq567}
\end{align}

We say that $y$ is \emph{exceptional} if, for some tuple $(\xi_1,\dots,\xi_9)$ with $\|\xi_i\|_1 \leq M$ for all $i$, one of the following is true:
\begin{enumerate}
\item there are at most $3$ values of $z$ for which $\xi_1 + \xi_4 + \xi_7 \overline{z^2} \in \Lambda_{\Psi}^+$, and $y$ is one of those values; or
\item there are at most $2$ values of $z$ for which $\xi_2 + \xi_5 + \overline{z}(\xi_4 + \xi_8) \in \Lambda_{\Psi}^+$, and $y$ is one of those values. 
\end{enumerate}
The number of exceptional $y$ is clearly $O(M^{9d})$. Suppose that $y$ is not exceptional, and that $\xi_1,\dots,\xi_9$ lies in the support of the sum \eqref{eq567}, thus \eqref{add-const} and \eqref{mult-const} hold. Then there are at least 3 values of $z$ for which for which $\xi_1 + \xi_4 + \xi_7 \overline{z^2} \in \Lambda_{\Psi}^{+}$. Take two of these, $z_1$ and $z_2$, for which $\overline{z_1^2} \neq \overline{z_2^2}$. Then $\xi_7(\overline{z_1^2} - \overline{z_2^2}) \in \Lambda_{\Psi}^+$, which implies (from the definition of $\Lambda_{\Psi}^+$) that $\xi_7 \in \Lambda_{\Psi}^+$. Therefore, from \eqref{add-const}, $\xi_1 + \xi_4 \in \Lambda_{\Psi}^+$ as well. Furthermore there are 2 values of $z$ for which $\xi_2 +\xi_5 + \overline{z}(\xi_4 + \xi_8) \in \Lambda_{\Psi}^+$. By much the same argument it follows that $\xi_2+ \xi_5, \xi_4 + \xi_8 \in \Lambda_{\Psi}^+$. It follows that if $y$ is not exceptional then \eqref{add-const} may be replaced by the stronger conclusion that 

\begin{equation}\label{add-const-stronger} \xi_1 + \xi_4 , \xi_7 , \xi_2 + \xi_5 , \xi_4 + \xi_8 \in \Lambda_{\Psi}^+. \end{equation}

Putting all this together, if $y$ is not exceptional (or $0$) then
\[
E(y)  = o_{p \rightarrow \infty}(M^3) + O(\eps M^4) + \sum_{\substack{\xi_1,\dots, \xi_9 \\ \xi_1 + \xi_4, \xi_7,\xi_2 + \xi_5, \xi_4 + \xi_8 \in \Lambda_{\Psi}^+ \\ \xi_3 + \xi_9, \xi_6 \in \Lambda_{\Psi}^{\times}}} \hat{F}(\xi_1,\xi_2,\xi_3) \hat{F}(\xi_4,\xi_5,\xi_6) \hat{F}(\xi_7, \xi_8, \xi_9).\]
We claim that the sum on the right here is precisely

\[ I(F) := \int F(t,u,v) F(t + u', u, v') F(t', u', v) d\mu_{H_{\Psi}}(t,u,v) d\mu_{H_{\Psi}}(t', u', v') .\] 

To see this, start with this latter expression and expand each copy of $F$ as a Fourier series. Recall also that $\mu_{H_{\Psi}} = \mu_{G_{\Psi}^+} \times \mu_{G_{\Psi}^+} \times \mu_{G_{\Psi}^{\times}}$. Then we get

\begin{align*}
I(F)& =\sum_{\xi_1,\dots,\xi_9} \hat{F}(\xi_1, \xi_2,\xi_3) \hat{F}(\xi_4,\xi_5,\xi_6) \hat{F}(\xi_7,\xi_8,\xi_9)\\
& \qquad \times \int e(\xi_1 \cdot t + \xi_2 \cdot u + \xi_4 \cdot (t + u') + \xi_5 \cdot u + \xi_7 \cdot t' + \xi_8 \cdot u') d\mu_{G_{\Psi}^+}^{\otimes 4}(t,u,t',u')\\
& \qquad \times \int v^{\xi_3 + \xi_9} (v')^{\xi_6} d\mu_{G_{\Psi}^{\times}}^{\otimes 2}(v,v').
\end{align*}
The claim now follows from the orthogonality relations \eqref{orth-plus}, \eqref{orth-times}.

Thus if $y$ is not exceptional (or $0$) then
\[ E(y) = o_{p \rightarrow \infty}(M^3) + O(\eps M^4) + I(F).\]
On the other hand, whatever the value of $y$ we have $|E(y)| \leq 1$ and the measure of the exceptional elements is at most $O(M^{9d}/p)$ since there are $O(M^{9d})$ of them.  It follows that
\begin{align*}
T(F \circ \Psi, 1_S , F \circ \Psi, F \circ \Psi) & = \E_y 1_S(y) E(y) \\ & = \mu(S) (I(F)+O(\epsilon M^4)) + o_{p \rightarrow \infty}(M^{9d})
\end{align*}
and the result is proved.
\end{proof}

\section*{Acknowledgments} 
This work was initiated at the workshop \emph{Combinatorics meets ergodic theory}, held at the Banff International Research Station (BIRS). The authors wish to thank BIRS for providing excellent working conditions. We would like to thank Roger Heath-Brown for a useful conversation about character sums.

\bibliographystyle{amsplain}


\begin{dajauthors}
\begin{authorinfo}[bg]
  Ben Green\\
The Mathematical Institute, Radcliffe Observatory Quarter, Woodstock Road, Oxford OX2 6GG\\
  ben.green\imageat{}maths\imagedot{}ox\imagedot{}ac\imagedot{}uk \\
\end{authorinfo}
\begin{authorinfo}[ts]
  Tom Sanders\\
The Mathematical Institute, Radcliffe Observatory Quarter, Woodstock Road, Oxford OX2 6GG\\
  tom.sanders\imageat{}maths\imagedot{}ox\imagedot{}ac\imagedot{}uk \\
\end{authorinfo}
\end{dajauthors}

\end{document}